\documentclass[a4paper,12pt,reqno]{amsart}
\usepackage{cmap}
\usepackage{amssymb}
\usepackage[T1]{fontenc}
\usepackage[utf8]{inputenc}
\usepackage[margin=1in]{geometry}
\usepackage{ifthen}
\usepackage{graphicx}
\usepackage{xcolor}
\usepackage[footnotesize]{caption}
\usepackage[bookmarksdepth=2]{hyperref}

\newcommand{\abs}[1]{\left| #1 \right|}
\newcommand{\expr}[1]{\left( #1 \right)}

\newcommand{\norm}[1]{\left\| #1 \right\|}
\newcommand{\set}[1]{\left\{ #1 \right\}}

\newcommand{\tscalar}[1]{\langle #1 \rangle}
\newcommand{\ind}{\mathbf{1}}

\newcommand{\sub}{\subseteq}
\newcommand{\C}{\mathbf{C}}
\newcommand{\R}{\mathbf{R}}

\newcommand{\A}{\mathcal{A}}
\newcommand{\form}{\mathcal{E}}

\newcommand{\D}{d}

\newcommand{\fourier}{\mathcal{F}}
\newcommand{\laplace}{\mathcal{L}}

\newcommand{\domain}{\mathcal{D}}
\newcommand{\eps}{\varepsilon}
\newcommand{\ph}{\varphi}

\newcommand{\thet}{\vartheta}

\newcommand{\ignore}[1]{}

\newcommand{\pvint}{\pv\!\!\int}

\newcommand{\odd}{\mathrm{odd}}
\newcommand{\even}{\mathrm{even}}

\newcommand{\formula}[2][nolabel]
{\ifthenelse{\equal{#1}{nolabel}}
 {\begin{align*} #2 \end{align*}}
 {\ifthenelse{\equal{#1}{}}
  {\begin{align} #2 \end{align}}
  {\begin{align} \label{#1} #2 \end{align}}
 }
}

\DeclareMathOperator{\pv}{pv}
\DeclareMathOperator{\Arg}{Arg}
\DeclareMathOperator{\imag}{Im}
\DeclareMathOperator{\real}{Re}
\DeclareMathOperator{\dist}{dist}

\DeclareMathOperator{\supp}{supp}

\theoremstyle{plain}
\newtheorem{theorem}{Theorem}[section]

\newtheorem{lemma}[theorem]{Lemma}
\newtheorem{corollary}[theorem]{Corollary}
\newtheorem{proposition}[theorem]{Proposition}

\theoremstyle{definition}
\newtheorem{definition}[theorem]{Definition}

\newtheorem{example}[theorem]{Example}
\newtheorem{remark}[theorem]{Remark}

\theoremstyle{remark}

\begin{document}

%
%

\title[Eigenvalues of pseudo-differential operators in an interval]{Asymptotic estimate of eigenvalues of pseudo-differential operators in an interval}
\author{Kamil Kaleta, Mateusz Kwa{\'s}nicki, Jacek Ma{\l}ecki}
\thanks{Work supported by NCN grant no. 2011/03/D/ST1/00311}
\address{Department of Pure and Applied Mathematics \\ Faculty of Fundamental Problems of Technology \\ Wroc{\l}aw University of Technology \\ ul. Wybrze{\.z}e Wyspia{\'n}\-skiego 27, 50-370 Wroc{\l}aw, Poland}
\email{kamil.kaleta@pwr.edu.pl, mateusz.kwasnicki@pwr.edu.pl, jacek.malecki@pwr.edu.pl}

\sloppy

\begin{abstract}
We prove a two-term Weyl-type asymptotic law, with error term $O(\tfrac{1}{n})$, for the eigenvalues of the operator $\psi(-\Delta)$ in an interval, with zero exterior condition, for complete Bernstein functions $\psi$ such that $\xi \psi'(\xi)$ converges to infinity as $\xi \to \infty$. This extends previous results obtained by the authors for the fractional Laplace operator ($\psi(\xi) = \xi^{\alpha/2}$) and for the Klein--Gordon square root operator ($\psi(\xi) = (1 + \xi^2)^{1/2} - 1$). The formula for the eigenvalues in $(-a, a)$ is of the form $\lambda_n = \psi(\mu_n^2) + O(\tfrac{1}{n})$, where $\mu_n$ is the solution of $\mu_n = \tfrac{n \pi}{2 a} - \tfrac{1}{a} \thet(\mu_n)$, and $\thet(\mu) \in [0, \tfrac{\pi}{2})$ is given as an integral involving~$\psi$.
\end{abstract}

\maketitle

%
%

\section{Introduction and statement of the results}

This is the final one in the series of articles where asymptotic formulae for eigenvalues of certain pseudo-differential operators in the interval are studied. The fractional Laplace operator $(-\Delta)^{\alpha/2}$ was considered in~\cite{bib:kkms10} for $\alpha = 1$ and in~\cite{bib:k12} for general $\alpha \in (0, 2)$, while in~\cite{bib:kkm13} the case of the Klein--Gordon square-root operator \mbox{$(-\Delta + 1)^{1/2} - 1$} was solved ($\Delta$~dentotes the second derivative operator, the Laplace operator in dimension one). In the present article we extend the above results to operators $\psi(-\Delta)$, where $\psi$ is an arbitrary complete Bernstein function such that $\xi \psi'(\xi)$ converges to infinity as $\xi \to \infty$.

Let $\lambda_n$ denote the nondecreasing sequence of eigenvalues of $\psi(-\Delta)$ in an interval $D = (-a, a)$, with zero condition in the complement of $D$. Furthermore, for $\mu > 0$ define
\formula[eq:thetmu]{
 \thet_\mu & = \frac{1}{\pi} \int_0^\infty \frac{\mu}{r^2 - \mu^2} \, \log \frac{\psi'(\mu^2) (\mu^2 - r^2)}{\psi(\mu^2) - \psi(r^2)} \, \D r .
}
We note that $\thet_\mu \in [0, \tfrac{\pi}{2})$ and $\tfrac{d}{d \mu} \thet_\mu = O(\tfrac{1}{\mu})$ as $\mu \to \infty$. Finally, let $\mu_n$ be a solution of
\formula[eq:mun]{
 \mu_n & = \tfrac{n \pi}{2 a} - \tfrac{1}{a} \thet_{\mu_n} .
}
We remark that the solution is unique for $n$ large enough, and
\formula{
 \mu_n & = \tfrac{n \pi}{2 a} - \tfrac{1}{a} \thet_{(n \pi) / (2 a)} + O(\tfrac{1}{n}) .
}
The following is the main result of the present article.

\begin{theorem}
\label{th:main}
If $\psi$ is a complete Bernstein function and $\lim_{\xi \to \infty} \xi \psi'(\xi) = \infty$, then
\formula[eq:main]{
 \lambda_n & = \psi(\mu_n^2) + O(\tfrac{1}{n}) && \text{as $n \to \infty$.}
}
\end{theorem}

In many cases, $\mu_n$ can be approximated with more explicit expressions, at the price of a weaker estimate of the error term. We provide two examples.

\begin{example}
Let $\psi(\xi) = \xi^{\alpha/2} + \xi^{\beta/2}$, where $0 < \beta < \alpha \le 2$. Then (see Example~\ref{ex:sum})
\formula{
 \thet_\mu & = \tfrac{(2 - \alpha) \pi}{8} + O(n^{\beta - \alpha}) , & \mu_n & = \tfrac{n \pi}{2 a} - \tfrac{(2 - \alpha) \pi}{8 a} + O(n^{\beta - \alpha}) , 
}
and consequently
\formula{
 \lambda_n & = (\tfrac{n \pi}{2 a} - \tfrac{(2 - \alpha) \pi}{8 a})^\alpha + (\tfrac{n \pi}{2 a} - \tfrac{(2 - \alpha) \pi}{8 a})^\beta + O(n^{\beta - 1}) .
}
\end{example}

\begin{example}
If $\psi$ is regularly varying at infinity with index $\tfrac{\alpha}{2} \in (0, 1]$, then one has $\lim_{\mu \to \infty} \thet_\mu = \tfrac{(2 - \alpha) \pi}{8}$ (see~\eqref{eq:thetreg}). Therefore,
\formula{
 \mu_n & = \tfrac{n \pi}{2 a} - \tfrac{(2 - \alpha) \pi}{8 a} + o(1) ,
}
and, using Karamata's monotone density theorem, one easily finds that
\formula{
 \lambda_n & = (1 - \tfrac{(2 - \alpha) \alpha}{4 n} + o(\tfrac{1}{n})) \psi((\tfrac{n \pi}{2 a})^2) .
}
\end{example}

\begin{remark}
The \emph{moderate growth condition} $\lim_{\xi \to \infty} \xi \psi'(\xi) = \infty$ is satisfied by all regularly varying functions with positive index. It is \emph{not} satisfied by a slowly varying complete Bernstein function $\psi(\xi) = \log(1 + \xi)$. There are, however, slowly varying functions which do satisfy the moderate growth condition, for example,
\formula{
 \psi(\xi) & = \int_1^\infty \frac{\xi}{\xi + z} \, \frac{\log z \, dz}{z} \, ,
}
which is asymptotically equal to $\tfrac{1}{2} (\log \xi)^2$ as $\xi \to \infty$.
\end{remark}

\begin{remark}
Numerical simulations for $\psi(\xi) = \xi^{\alpha/2}$ (using the results of~\cite{bib:dkk15}) strongly suggest that the error in~\eqref{eq:main} is in fact of order $O(\tfrac{1}{n^2})$. There is no numerical evidence for more general functions $\psi$. We believe that at least when $\psi(\xi) = \sqrt{\xi}$, one can use a method applied in a somewhat similar problem in~\cite{bib:d70} to obtain a version of~\eqref{eq:main} with an additional term in the asymptotic expansion and an improved bound of the error term, but this is far beyond the scope of the present article.
\end{remark}

We point out that relatively little is known about $\lambda_n$. Most results, including all listed below, cover also higher-dimensionsinal domains, but provide significantly less detailed information. Extension of Theorem~\ref{th:main} for higher-dimensional domains seems out of reach with the present methods.

Best known estimates of $\lambda_n$, proved in~\cite{bib:cs05}, are given in terms of the corresponding eigenvalues $\lambda_n^\Delta$ of the Laplace operator $-\Delta$, namely
\formula{
 C \psi(\lambda_n^\Delta) \le \lambda_n & \le \psi(\lambda_n^\Delta) ;
}
a more direct statement for the case of an interval is given in~\eqref{eq:cs} below. First term of the asymptotic expansion of $\lambda_n$, namely $\lambda_n \sim \psi(\lambda_n^\Delta)$, is given in many cases in~\cite{bib:bg59}. This result follows by a Tauberian theorem from the asymptotic expression for the \emph{trace} $\sum_{n = 1}^\infty e^{-t \lambda_n}$ as $t \to 0^+$.

Second term of the asymptotic expansion of the trace has been found in~\cite{bib:bk08,bib:bks09} for $(-\Delta)^{\alpha/2}$, in~\cite{bib:bmn14,bib:ps14} for $(-\Delta + 1)^{\alpha/2} - 1$, and finally in~\cite{bib:bs15} for a rather general class of isotropic Lévy processes with unimodal Lévy measure, satisfying some mild regularity conditions. Tauberian theory is, however, insufficient to obtain a result similar to Theorem~\ref{th:main} from the two-term expansion of the trace. To the knowledge of the authors, no results of this kind are known for domains other than intervals, with the only exception of the well-studied classical situation of the Laplace operator $-\Delta$. The only related result, proved in~\cite{bib:fg14}, provides a two-term asymptotic expansion of Cesàro means $\tfrac{1}{N} \sum_{n = 1}^N \lambda_n$ for $(-\Delta)^{\alpha/2}$ using the methods of semi-classical analysis.

The proof of Theorem~\ref{th:main} is based on the explicit expression for the generalised eigenfunctions of the operator $\psi(-\Delta)$ in the half-line, found in~\cite{bib:kkms10} for $(-\Delta)^{1/2}$, and in~\cite{bib:k11,bib:kmr13} for $\psi(-\Delta)$ for a general complete Bernstein function $\psi$. The asymptotic expression~\eqref{eq:main} for $(-\Delta)^{\alpha/2}$ simplifies to
\formula{
 \lambda_n & = (\tfrac{n \pi}{2 a} - \tfrac{(2 - \alpha) \pi}{8 a})^\alpha + O(\tfrac{1}{n}) && \text{as $n \to \infty$,}
}
because $\thet_\mu = \tfrac{(2 - \alpha) \pi}{8}$. As mentioned above, this was proved for $\alpha = 1$ in~\cite{bib:kkms10}, with constant $1$ in the asymptotic notation $O(\tfrac{1}{n})$, and for general $\alpha \in (0, 2)$ in~\cite{bib:k12}, with a rather big constant in the term $O(\tfrac{1}{n})$. A very careful estimate of~\cite{bib:kkm13} yielded a version of~\eqref{eq:main} uniform in $a > 0$ for the operator $(-\Delta + 1)^{1/2} - 1$. In the present article we do not pay attention to the constant in the asymptotic term $O(\tfrac{1}{n})$. All our estimates are, however, explicit, and so it is theoretically possible to trace the dependence of this constant on $a$ and $\psi$.

To facilitate the reading of the article, we sketch the main idea of the proof. The generalised eigenfunction of $\psi(-\Delta)$ in the half-line $(0, \infty)$ corresponding to the eigenvalue $\psi(\mu^2)$ is given by an explicit formula $F_\mu(x) = \sin(\mu x + \thet_\mu) - G_\mu(x)$, where $G_\mu$ is the Laplace transform of a certain non-negative measure (here `generalised' essentially means `not square integrable'). We construct approximation $\tilde{\ph}_n$ to eigenfunctions of $\psi(-\Delta)$ in $(-a, a)$ by interpolating between $F_\mu(a + x)$ near $-a$ and $\pm F_\mu(a - x)$ near $a$. In order that the sine terms agree, we need to set $\mu = \mu_n$ defined in~\eqref{eq:mun}. Due to non-locality of $\psi(-\Delta)$, $\tilde{\ph}_n$ is not an eigenfunction; nevertheless, we show that the $L^2(D)$ distance of $\psi(-\Delta) \tilde{\ph}_n$ and $\mu_n \tilde{\ph}_n$ does not exceed $O(\tfrac{1}{n})$ (Lemma~\ref{lem:approxnorm0}). This is sufficient to prove that there is some eigenvalue $\lambda_{k(n)}$ within the $O(\tfrac{1}{n})$ range from $\psi(\mu_n^2)$. Using the assumption that $\xi \psi'(\xi)$ diverges to infinity as $\xi \to \infty$, one easily finds that the numbers $k(n)$ are distinct for sufficiently large $n$. It remains to estimate the number of eigenvalues $\lambda_j$ not counted as $\lambda_{k(n)}$: this turns out to follow from an estimate of the trace (Lemma~\ref{lem:trace}).

We conjecture that~\eqref{eq:main} holds for arbitrary complete Bernstein functions, without the moderate growth condition $\lim_{\xi \to \infty} \xi \psi'(\xi) = \infty$. Note that, however, if this growth condition is not satisfied (for example, when $\psi(\xi) = \log(1 + \xi)$) and $a$ is large enough, then one cannot expect that the numbers $k(n)$ are distinct. Therefore, an extension of Theorem~\ref{th:main} to general complete Bernstein function would require a completely different approach. It is also natural to expect that~\eqref{eq:main} holds for more general functions $\psi$, for example, for all Bernstein functions $\psi$ satisfying the growth condition. However, no expressions for the generalised eigenfunctions $F_\mu$ are known unless $\psi$ is a complete Bernstein function, and so our approach cannot currently be used in this case.

The method described above has been designed in~\cite{bib:kkms10} and sucessfully used in~\cite{bib:k12} and~\cite{bib:kkm13}. The core of the argument remains the same in the present article. Nevertheless, proving Theorem~\ref{th:main} in this generality requires rather non-obvious estimates of $\thet_\mu$, $\tfrac{d}{d \mu} \thet_\mu$ and $G_\mu(x)$, as well as many other modfications; for example, the trace estimate in the final part of the proof needed some improvements. 

The remaining part of the article is divided into two sections. In Preliminaries, we recall definitions and basic properties of complete Bernstein functions (Section~\ref{sec:precbf}) and the operator $\psi(-\Delta)$ in full space $\R$ and in (bounded or unbounded) intervals (Sections~\ref{sec:preop}--\ref{sec:preopdom}). We also recall known properties of the eigenvalues $\lambda_n$ (Section~\ref{sec:prelambda}) and the generalised eigenfunctions $F_\mu(x)$ (Section~\ref{sec:prefmu}). Finally, we prove the necessary estimates of $\thet_\mu$ (Section~\ref{sec:prethetmu}) and $G_\mu(x)$ (Section~\ref{sec:pregmu}). The proof of Theorem~\ref{th:main} is given in Section~\ref{sec:proofs}, which is divided into five parts. Pointwise estimates for $\psi(-\Delta)$ (Section~\ref{sec:prpoint}) are taken from~\cite{bib:kkm13}. Construction of approximations to eigenfunctions $\tilde{\ph}_n$ is followed by a technical lemma, which asserts that $\tilde{\ph}_n$ is in the domain of $\psi(-\Delta)$ in the interval (Section~\ref{sec:prconstr}). An estimate for $\psi(-\Delta) \tilde{\ph}_n$ (Section~\ref{sec:prest}) follow then easily from the results given in Preliminaries. This is used to find estimates for the eigenvalues $\lambda_{k(n)}$ (Section~\ref{sec:prlambda}). We conclude the proof by showing that $k(n) = n$ for $n$ large enough (Section~\ref{sec:prtrace}). As a side-result, we obtain some properties of the eigenfunctions, listed in the final Section~\ref{subsec:prop}.

%
%

\section{Preliminaries}
\label{sec:pre}

All functions considered below are Borel measurable. 
For $p \in [1, \infty)$ and an open set $D \sub \R$, the Lebesgue space $L^p(D)$ is the set of functions $f$ on $D$ such that $\|f\|_{L^p(D)} = (\int_D |f(x)|^p dx)^{1/p}$ is finite, and $f \in L^\infty(D)$ if and only if the essential supremum $\|f\|_{L^\infty(D)}$ of $|f(x)|$ over $x \in D$ is finite. 
The space of smooth functions with compact support contained in $D$ is denoted by $C_c^\infty(D)$. By $C_0(D)$ we denote the space of continuous functions in $\R$ which are equal to $0$ in $\R \setminus D$ and which satisfy the condition $\lim_{x \to \pm \infty} f(x) = 0$.

The Fourier transform of a function $f \in L^2(\R)$ is denoted by $\fourier f$. If $f \in L^2(\R) \cap L^1(\R)$, then $\fourier f(\xi) = \int_{-\infty}^\infty f(x) e^{-i \xi x} dx$. The Laplace transform of a function $f$ is denoted by $\laplace f$, $\laplace f(\xi) = \int_0^\infty f(x) e^{-\xi x} dx$. Symbols $x$, $y$, $z$ are used for spatial variables, while $\xi$, $\eta$, $\mu$ typically correspond to `Fourier space' variables.

We sometimes use standard asymptotic notation: we write $f(n) = O(g(n))$ if $\limsup_{n \to \infty} |f(n) / g(n)| < \infty$, and $f(n) = o(g(n))$ if $\lim_{n \to \infty} |f(n) / g(n)| = 0$.

\subsection{Complete Bernstein functions}
\label{sec:precbf}

In this section we recall several classical definitions. A function $f(x)$ on $(0, \infty)$ is said to be \emph{completely monotone} if $(-1)^n f^{(n)}(x) \ge 0$ for all $x > 0$ and $n = 0, 1, 2, \dots$ By Bernstein's theorem (\cite[Theorem~1.4]{bib:ssv10}), $f$ is completely monotone if and only if it is the Laplace transform of a (possibly infinite) Radon measure on $[0, \infty)$. If $f$ is nonnegative on $(0, \infty)$ and $f'$ is completely monotone, then $f$ is said to be a \emph{Bernstein function}. By Bernstein's theorem, Bernstein functions have the representation
\formula[eq:bf]{
 f(x) & = c x + \tilde{c} + \int_{(0, \infty)} (1 - e^{-z x}) M(dz)
}
for some $c, \tilde{c} \ge 0$ and a Radon measure $M$ such that $\int_{(0, \infty)} \min(z, 1) M(dz) < \infty$. The above formula extends to complex $x$ such that $\real x \ge 0$, and defines a continuous function holomorphic in the region $\real x > 0$.

If the measure $M$ in~\eqref{eq:bf} is absolutely continuous with respect to the Lebesgue measure, and the density function is completely monotone, then $f$ is said to be a \emph{complete Bernstein function}. One easily verifies that in this case
\formula[eq:cbf]{
 f(x) & = c x + \tilde{c} + \frac{1}{\pi} \int_{(0, \infty)} \frac{x}{x + z} \, \frac{m(dz)}{z}
}
for some $c, \tilde{c} \ge 0$ and a Radon measure $m$ such that $\int_{(0, \infty)} \min(1/z, 1/z^2) m(dz) < \infty$. The above formula defines a holomorphic extension of $f$ in the region $\C \setminus (-\infty, 0]$.

Bernstein and complete Bernstein functions appear in a number of different areas of mathematics. For more information on these objects, we refer the reader to~\cite{bib:ssv10}.

We will need the following technical result, proved in the Appendix.

\begin{lemma}
\label{lem:cbfm}
Let $f$ is a complete Bernstein function with representation~\eqref{eq:cbf}. Let $g$ be a holomorphic function in $\{w \in \C : |\Arg w| < C_1 \}$ (with $0 < C_1 < \tfrac{\pi}{2}$) such that $g(x)$ is real for $x > 0$, and let $h$ be a continuous function on $(0, \infty)$. Denote
\formula{
 G(y) & = \sup_{\substack{y/4 \le |z| \le 4 y \\ |\Arg z| < C_1}} |g(z)| , & H(y) & = \sup_{y/4 \le x \le 4 y} |h(x)| ,
}
and suppose that
\formula{
 G(x) H(x) & \le C_2 \min(x^{-1}, x^{-2}) , & C_3 = \int_0^\infty (1 + y) G(y) H(y) dy & < \infty
}
for $x > 0$. Then
\formula[eq:fgh]{
 \int_{(0, \infty)} g(x) h(x) m(dx) & = \lim_{\eps \to 0^+} \frac{1}{\pi} \int_0^\infty \imag (f(-e^{-i \eps} x) g(e^{-i \eps} x)) h(x) dx \\
 & = \lim_{\eps \to 0^+} \frac{1}{\pi} \int_0^\infty \imag (f(-e^{-i \eps} x)) g(x) h(x) dx .
}
\end{lemma}

Following~\cite{bib:k11,bib:kmr13}, if $\psi$ is a non-constant complete Bernstein function such that $\psi(0) = 0$, and $\mu > 0$, we denote
\formula{
 \psi_\mu(\xi) & = \frac{1 - \xi / \mu^2}{1 - \psi(\xi) / \psi(\mu^2)}
}
for $\xi \in \C \setminus ((-\infty, 0] \cup \{\mu^2\})$, and $\psi_\mu(\mu^2) = \psi(\mu^2) / (\mu^2 \psi'(\mu^2))$. We also let
\formula{
 \psi^\dagger(\xi) & = \exp \expr{\frac{1}{\pi} \int_0^\infty \frac{\xi}{\xi^2 + r} \, \log \psi(r^2) dr}
}
for $\xi \in \C$ with $\real \xi > 0$. Then $\psi_\mu$ is a complete Bernstein function, $\psi^\dagger$ extends to a complete Bernstein function, and we have the Wiener--Hopf identity $\psi^\dagger(\xi) \psi^\dagger(-\xi) = \psi(-\xi^2)$ for $\xi \in \C \setminus \R$, see, for example, \cite[Proposition~2.19 and Lemma~3.8]{bib:k11}. Finally, we denote $\psi^\dagger_\mu = (\psi_\mu)^\dagger$.

In principle we could extend the definition of $\psi_\mu$ to general non-constant complete Bernstein functions $\psi$, so that $\psi_\mu(\xi) = (1 - \xi^2 / \mu^2) / (1 - (\psi(\xi) - \psi(0)) / (\psi(\mu^2) - \psi(0)))$. All results proved below hold true with this definition. However, to keep the notation simpler, we will typically assume that $\psi(0) = 0$. For brevity, we also denote $\psi(\infty) = \lim_{\xi \to \infty} \psi(\xi) \in [0, \infty]$. 

Unless otherwise stated, in what follows we assume that $\psi(\xi)$ is a non-constant complete Bernstein function which satisfies $\psi(0) = 0$, that is, $\tilde{c} = 0$ in representation~\eqref{eq:cbf} for $\psi$.

\subsection{Analytical definition of the operator $\psi(-\Delta)$ in $\R$ and in intervals}
\label{sec:preop}

In this section it is enough to assume that $\psi$ is an increasing, nonnegative function on $[0, \infty)$, which satisfies
\formula[eq:powertype]{
 1 + \psi(\xi + \eta) \le C (1 + \eta)^\alpha (1 + \psi(\xi))
}
for all $\xi \ge \eta \ge 0$ and some $C, \alpha \ge 1$. When $\psi$ is a complete Bernstein function, then~\eqref{eq:powertype} holds with $\alpha = 1$ and $C = 1$, because $\psi(\xi + \eta) \le \psi(\xi) + \psi(\eta) \le \psi(\xi) + C (1 + \eta)$.

The operator $A = \psi(-\Delta)$ is an unbounded, non-local, self-adjoint operator on $L^2(\R)$, defined as follows. The domain $\domain(A)$ of $A$ consists of functions $f \in L^2(\R)$ such that $(1 + \psi(\xi^2)) \fourier f(\xi)$ is square integrable. Clearly, $\domain(A)$ contains $C_c^\infty(\R)$. For $f \in \domain(A)$,
\formula{
 \fourier A f(\xi) & = \psi(\xi^2) \fourier f(\xi) .
}
In other words, $A$ is a \emph{Fourier multiplier} with symbol $\psi(\xi^2)$. This explains the notation $A = \psi(-\Delta)$: the second derivative operator $\Delta$ is a Fourier multiplier with symbol $-\xi^2$. Furthermore, by Plancherel's theorem, $A$ is positive-definite.

Let $\domain(\form)$ denote the space of $f \in L^2(\R)$ such that $(1 + \psi(\xi^2))^{1/2} \fourier f(\xi)$ is square integrable. For $f, g \in \domain(\form)$ the quadratic form $\form(f, g)$ associated to $A$ is defined by
\formula{
 \form(f, g) & = \frac{1}{2 \pi} \int_{\R} \psi(\xi^2) \fourier f(\xi) \fourier g(\xi) d \xi .
}
The inner product $\form_1(f, g) = \tscalar{f, g} + \form(f, g)$ makes $\domain(\form)$ into a Hilbert space. If $f \in \domain(A)$, then $\form(f, g) = \tscalar{A f, g}$, and $\domain(A)$ is a dense subset of the Hilbert space $\domain(\form)$.

Let $D$ be an open subset of $\R$. The following definition states that the operator $A_D$ is the \emph{Friedrichs extension} (or the \emph{minimal self-adjoint extension}) of the restriction of $A$ to $C_c^\infty(D)$.

\begin{definition}
\label{def:h}
The domain $\domain(\form_D)$ of the form $\form_D$ is the closure of $C_c^\infty(D)$ in the Hilbert space $\domain(\form)$, and $\form_D(f, g) = \form(f, g)$ for $f, g \in \domain(\form_D)$. The operator $A_D$ is associated to the form $\form_D$: $f \in \domain(\form_D)$ is in the domain $\domain(A_D)$ of $A_D$ if and only if there is a function $A_D f \in L^2(D)$ such that $\form(f, g) = \tscalar{A_D f, g}$ for $g \in \domain(\form_D)$ (or, equivalently, for $g \in C_c^\infty(D)$).
\end{definition}

The following result is well-known in the context of general Dirichlet forms and generators of L\'evy processes, see~\cite{bib:fot10, bib:s99} for more general results in this direction. For completeness, we provide a short proof.

\begin{proposition}[{see~\cite[Proposition~2.2]{bib:kkm13}}]
\label{prop:domain}
If $D$ is a bounded interval, then $f \in \domain(\form_D)$ if and only if $f \in \domain(\form)$ and $f = 0$ almost everywhere in $\R \setminus D$.
\end{proposition}

\begin{proof}
By definition, if $f \in \domain(\form_D)$, then $f \in \domain(\form)$ and $f = 0$ almost everywhere in $\R \setminus D$. Let $f \in \domain(\form)$ and $f = 0$ almost everywhere in $\R \setminus D$. The result follows from the following claim: there is a sequence $f_n \in C_c^\infty(D)$ such that
\formula{
 \form_1(f_n - f, f_n - f) & = \frac{1}{2 \pi} \int_{-\infty}^\infty (1 + \psi(\xi^2)) |\fourier f_n(\xi) - \fourier f(\xi)|^2 d\xi
}
converges to $0$ as $n \to \infty$.

Let $h_n \in C_c^\infty(\R^D)$ be an approximation to the identity such that $h_n(x) = n h(n x)$, $h(x) \ge 0$, $\int_\R h(x) dx = 1$ and $h(x) = 0$ for $x \notin (-1, 1)$. Note that $h_n$ is zero outside $(-\tfrac{1}{n}, \tfrac{1}{n})$. Let
\formula{
 g_n(x) & = h_n * f(x) , & f_n(x) & = g_n((x - b_n) / a_n) ,
}
where $(x - b_n) / a_n$ maps the $\tfrac{2}{n}$-neighbourhood of $I$ into $I$, with $a_n \ge 1$, $\lim_{n \to \infty} a_n = 1$ and $\lim_{n \to \infty} b_n = 0$. Observe that $f_n \in C_c^\infty(D)$ and
\formula{
 \fourier f_n(\xi) & = a_n e^{-i b_n \xi} \fourier g_n(a_n \xi) = a_n e^{-i b_n \xi} \fourier h(\tfrac{1}{n} (a_n \xi)) \fourier f(a_n \xi) .
}
Since $f, g \in L^1(\R)$, $\fourier f$ and $\fourier h$ are continuous. Furthermore, $\fourier h(0) = 1$ and $|\fourier h(\xi)| \le 1$ for $\xi \in \R$. It follows that $\fourier f_n$ converges pointwise to $\fourier f$, and for $n$ large enough,
\formula{
 |\fourier f_n(\xi)| & \le 2 |\fourier f(a_n \xi)|
}
for all $\xi \in \R$. Hence, if $u(\xi) = (1 + \psi(\xi^2)) |\fourier f(\xi)|^2$, then for $n$ large enough,
\formula{
 (1 + \psi(\xi^2)) |\fourier f_n(\xi) - \fourier f(\xi)|^2 & \le 2 (1 + \psi(\xi^2)) (|\fourier f_n(\xi)|^2 + |\fourier f(\xi)|^2) \\
 & \le 4 u(a_n \xi) + 2 u(\xi)
}
for all $\xi$. By the assumption, $u(\xi)$ is integrable. Therefore, the family of functions $(1 + \psi(\xi^2)) |\fourier f_n(\xi) - \fourier f(\xi)|^2$ is tight and uniformly integrable. By the Vitali's convergence theorem, $\form_1(f_n - f, f_n - f)$ converges to $0$ as $n \to \infty$, as desired.
\end{proof}

We remark that the above result in general fails to be true for arbitrary open sets $D$. It is, in particular, not true when $D = \R \setminus \{0\}$ and $\psi(\xi) = \xi^{\alpha/2}$ with $\alpha \in (1, 2]$.

\subsection{Markov semigroup generated by $A$}
\label{sec:preopfull}

From now on, $\psi$ is a complete Bernstein function. The operator $-A = -\psi(-\Delta)$ generates a strongly continuous semigroup of self-adjoint contractions
\formula{
 T(t) & = \exp(-t A) ,
}
where $t \ge 0$. Note that $T(0)$ is the identity operator, $T(t)$ is the Fourier multiplier with symbol $\exp(-t \psi(\xi^2))$, and
\formula[eq:agen]{
\begin{aligned}
 \domain(A) & = \set{f \in L^2(\R) : \lim_{t \to 0^+} \frac{T(t) f - f}{t} \text{ exists in } L^2(\R)} , \\ -A f & = \lim_{t \to 0^+} \frac{T(t) f - f}{t} \, .
\end{aligned}
}
For $t > 0$, the operator $T(t)$ is a convolution operator: for all $f \in L^2(\R)$,
\formula[eq:t0]{
 T(t) f(x) & = \int_{-\infty}^\infty f(x - y) T(t; dy) ,
}
where $T(t; dx)$ is a sub-probability measure with characteristic function $\exp(-t \psi(\xi^2))$ and total mass $e^{-t \psi(0)}$. Furthermore,
\formula{
 T(t; dx) = e^{-t \psi(\infty)} \delta_0(dx) + T(t; x) dx ,
}
where $T(t; x) = T(t; -x)$ is a decreasing function of $x > 0$ (see~\cite{bib:ssv10}). Hence, $T(t)$ is a \emph{Markov operator}, and formula~\eqref{eq:t0} defines a contraction on every $L^p(\R)$ ($p \in [1, \infty]$), and also on $C_0(\R)$. In each of these Banach spaces, the generator of the semigroup $T(t)$ is defined in a similar way as in~\eqref{eq:agen}; for example, 
\formula{
\begin{aligned}
 \domain(A; C_0(\R)) & = \set{f \in C_0(\R) : \lim_{t \to 0^+} \frac{T(t) f - f}{t} \text{ exists in } C_0(\R)} , \\
 -A f & = \lim_{t \to 0^+} \frac{T(t) f - f}{t} \, .
\end{aligned}
}
Since the above definitions of $A f$ are consistent on the intersections of domains with limits in different function spaces: $L^p(\R)$ for $p \in [1, \infty]$ or $C_0(\R)$, we abuse the notation and use the same symbol $-A$ for the generator of the semigroup $T(t)$ in any of these spaces. Observe that $C_c^\infty(\R)$ is contained in $\domain(A, L^p(\R))$ ($p \in [1, \infty]$) and in $\domain(A; C_0(\R))$, and it is the core of $A$ in each of these Banach spaces except $L^\infty(\R)$ (see~\cite{bib:a04,bib:j01}). Whenever we write $\domain(A)$, we mean $\domain(A; L^2(\R))$.

If $\psi(\xi)$ has the representation given in~\eqref{eq:cbf}, then for $f \in C_c^\infty(\R)$ we have
\formula[eq:gen]{
 A f(x) & = -c f''(x) + \tilde{c} f(x) + \pvint_{-\infty}^\infty (f(x) - f(y)) \nu(x - y) dz ,
}
where by the subordination formula,
\formula{
 \nu(z) & 
 = \frac{1}{2 \pi} \int_{(0, \infty)} e^{-|z| \zeta^{1/2}} \, \frac{m(d\zeta)}{\zeta^{1/2}} \, ,
}
and `$\pvint$' denotes the Cauchy principal value integral:
\formula{
 \pvint_{-\infty}^\infty (f(x) - f(x + z)) \nu(z) dz & = \lim_{\eps \to 0^+} \int_{\R \setminus (-\eps, \eps)} (f(x) - f(x + z)) \nu(z) dz ;
}
see, for example, \cite[Proposition~2.14]{bib:k11}.

\subsection{Markov semigroup generated by $A_D$}
\label{sec:preopdom}

Let $D$ be a (possibly unbounded) interval. The operator $-A_D$ generates a strongly continuous semigroup of operators
\formula{
 T_D(t) & = \exp(-t A_D) .
}
The operators $T_D(t)$ are given by
\formula{
 T_D(t) f(x) & = \int_D f(y) T_D(t; x, dy) ,
}
where
\formula{
 T_D(t; x, dy) & = e^{-t \psi(\infty)} \delta_x(dy) + T_D(t; x, y) dy .
}
It is known that $0 \le T_D(t; x, y) \le T(t, x - y)$, and we let $T_D(t; x, y) = 0$ whenever $x \notin D$ or $y \notin D$. Hence, $T_D(t)$ form a contraction semigroup on each of the spaces $L^p(D)$ ($p \in [1, \infty]$), and if $\psi$ is unbounded, then also on $C_0(D)$ (see~\cite{bib:c85,bib:k11,bib:s99}). The generator of each of these semigroups is again denoted by $-A_D$, and it acts on an appropriate domain $\domain(A_D; L^p)$ or $\domain(A_D; C_0)$.

\subsection{Spectral theory for $A_{(-a, a)}$}
\label{sec:prelambda}

Suppose that $D$ is a bounded interval and that $\exp(-2 t \psi(\xi^2))$ is integrable for some $t > 0$. Then $T_D(t; x, y)$ is a Hilbert--Schmidt kernel, and so $T_D(t)$ is a compact operator on $L^2(D)$. Hence, there is a complete orthonormal set of eigenfunctions $\ph_n \in L^2(D)$ of $T_D(t)$. By strong continuity and the semigroup property, the eigenfunctions do not depend on $t > 0$, and the corresponding eigenvalues have the form $e^{-t \lambda_n}$ for all $t > 0$, where the sequence $\lambda_n$ is nondecreasing and converges to $\infty$. 


By translation invariance, with no loss of generality we may assume that $D = (-a, a)$. By symmetry, $T_D(t; x, y) = T_D(t; -x, -y)$, and hence the spaces of odd and even $L^2(D)$ functions are invariant under the action of $T_D(t)$. Therefore, we may assume that every $\ph_n$ is either an odd or an even function. The \emph{ground state eigenvalue} $\lambda_1$ is positive and simple (unless $\psi$ is constant), and the corresponding \emph{ground state eigenfunction} has constant sign in $D$; we choose it to be positive in $D$. The functions $\ph_n$ are also the eigenfunctions of $A_D$ (because $-A_D$ is the generator of the semigroup $T_D(t)$), and $\lambda_n$ are the corresponding eigenvalues.

No closed-form expression for $\lambda_n$ and $\ph_n$ is available, except when $\psi(\xi) = c \xi + \tilde{c}$. By a general result of~\cite{bib:bg59} (see Theorem~2.3 therein), $\lambda_n \sim \psi((\tfrac{n \pi}{2 a})^2)$ as $n \to \infty$ (the original statement includes only the case when $\psi(\xi) \sim \xi^{\alpha/2}$ for some $\alpha \in (0, 2)$, but it can be easily extended to more general $\psi$). Best known general estimates of $\lambda_n$ are found in~\cite[Theorem~4.4]{bib:cs05}, where it is proved that:
\formula[eq:cs]{
 \frac{1}{2} \, \psi\!\expr{\expr{\frac{n \pi}{2 a}}^{\!2}} \le \lambda_n & \le \psi\!\expr{\expr{\frac{n \pi}{2 a}}^{\!2}} .
}
Note that the upper bound in~\eqref{eq:cs} follows relatively easily from the \emph{operator monotonicity} of $\psi$: the form associated to $A_D$ is bounded above by the form of $\psi(-\Delta_D)$, and the eigenvalues of the latter are equal to $\psi((\tfrac{n \pi}{2 a})^2)$. The proof of the lower bound is more intricate.

\subsection{Spectral theory for $A_{(0, \infty)}$}
\label{sec:prefmu}

The spectrum of $A_D$ for an unbounded interval $D$ is continuous. When $D = \R$, then $A_D = A$ takes diagonal form in the Fourier space, and $e^{i \xi x}$ ($\xi \in \R$) are the $L^\infty$ eigenfunctions of $A$. Similar eigenfunction expansion was obtained for the half-line using an appropriate version of the Wiener--Hopf method in~\cite{bib:k11,bib:kmr13}. Due to translation invariance and symmetry, it suffices to consider $D = (0, \infty)$.

\begin{definition}
\label{def:halfline}
Suppose that $\psi$ is a non-constant complete Bernstein function such that $\psi(0) = 0$. For $x, \mu > 0$, let
\formula{
 F_\mu(x) & = \sin(\mu x + \thet_\mu)  - G_\mu(x) ,
}
where $\thet_\mu \in [0, \tfrac{\pi}{2})$ and $G_\mu$ is a completely monotone function on $(0, \infty)$. More precisely,
\formula{
 \thet_\mu & = \frac{1}{\pi} \int_0^\infty \frac{\mu}{r^2 - \mu^2} \, \log \frac{\psi'(\mu^2) (\mu^2 - r^2)}{\psi(\mu^2) - \psi(r^2)} \, \D r
}
(as in~\eqref{eq:thetmu}), and $G_\mu$ is the Laplace transform of a measure $\gamma_\mu$,
\formula{
 G_\mu(x) & = \laplace \gamma_\mu(x) = \int_{(0, \infty)} e^{-x \xi} \gamma_\mu(\D \xi) ,
}
with
\formula{
 \gamma_\mu(\D \xi) & = \frac{1}{\pi} \lim_{\eps \to 0^+} \imag\!\expr{\frac{\mu \psi'(\mu^2)}{\psi(\mu^2) - \psi(-e^{-i \eps} \xi^2)}} \times \\
 & \hspace*{8em} \times \exp\!\expr{-\frac{1}{\pi} \int_0^\infty \frac{\xi}{\xi^2 + r^2} \, \log \frac{\psi'(\mu^2) (\mu^2 - r^2)}{\psi(\mu^2) - \psi(r^2)} \, dr} d\xi
}
for $\mu, \xi, x > 0$.
\end{definition}

Equivalently, $F_\mu(x)$ is defined by its Laplace transform: for $\xi \in \C$ with $\real \xi > 0$,
\formula{
 \laplace F_\mu(\xi) & = \int_0^\infty F_\mu(x) e^{-\xi x} dx = \frac{\mu}{\mu^2 + \xi^2} \, \exp\!\expr{\frac{1}{\pi} \int_0^\infty \frac{\xi}{\xi^2 + r^2} \, \log \frac{\psi'(\mu^2) (\mu^2 - r^2)}{\psi(\mu^2) - \psi(r^2)} \, dr\!} ,
}
see~\cite[Theorem~1.3]{bib:kmr13} and~\cite[Theorem~1.1]{bib:k11}. We have the short-hand expressions
\formula{
 \laplace F_\mu(\xi) & = \frac{\mu}{\mu^2 + \xi^2} \, \frac{\psi^\dagger_\mu(\xi)}{\sqrt{\psi_\mu(\mu^2)}} \, , \\
 \thet_\mu & = \Arg \psi^\dagger_\mu(i \mu) , \\
 \gamma_\mu(d \xi) & = \frac{1}{\pi} \lim_{\eps \to 0^+} \frac{\mu}{\mu^2 + \xi^2} \, \frac{\imag \psi_\mu(-e^{-i \eps} \xi^2)}{\sqrt{\psi_\mu(\mu^2)} \, \psi^\dagger_\mu(\xi)} \, d\xi ,
}
again see~\cite[Remark~4.12]{bib:k11} (note the typo in formula~(4.14) therein) and~\cite[formulae~(2.17)--(2.19)]{bib:kmr13}. The expressions for $\gamma_\mu(d\xi)$ given above are slightly different than in~\cite{bib:k11, bib:kmr13}, so we provide a short justification. By Lemma~\ref{lem:cbfm} and the identity
\formula{
 \laplace G_\mu(\xi) & = \frac{\mu \cos \thet_\mu + \xi \sin \thet_\mu}{\mu^2 + \xi^2} - \laplace F_\mu(\xi) ,
}
we have
\formula{
 \gamma_\mu(d \xi) & = \tfrac{1}{\pi} \lim_{\eps \to 0^+} \imag(\laplace G_\mu(e^{-i \eps} \xi)) d\xi = -\tfrac{1}{\pi} \lim_{\eps \to 0^+} \imag(\laplace F_\mu(e^{-i \eps} \xi)) d\xi .
}
The expression for $\laplace F_\mu(\xi)$ and the Wiener--Hopf identity $\psi^\dagger_\mu(\xi) \psi^\dagger_\mu(-\xi) = \psi_\mu(-\xi^2)$ give
\formula{
 \gamma_\mu(d \xi) & = \frac{1}{\pi} \lim_{\eps \to 0^+} \imag\!\expr{\frac{\mu}{\mu^2 + e^{-2 i \eps} \xi^2} \, \frac{\psi_\mu(-e^{-2 i \eps} \xi^2)}{\sqrt{\psi_\mu(\mu^2)} \, \psi^\dagger_\mu(e^{-i \eps} \xi)}} d\xi \\
 & = \frac{1}{\pi} \lim_{\eps \to 0^+} \frac{\mu}{\mu^2 + \xi^2} \, \frac{\imag \psi_\mu(-e^{-i \eps} \xi^2)}{\sqrt{\psi_\mu(\mu^2)} \, \psi^\dagger_\mu(\xi)} \, d\xi ,
}
as desired; here we used Lemma~\ref{lem:cbfm} again.

We extend the definition of $F_\mu$ and $G_\mu$ to $\R$ so that $F_\mu(x) = G_\mu(x) = 0$ for $x \le 0$. The functions $F_\mu$ ($\mu > 0$) are $L^\infty$ eigenfunctions of $A_D$ and play a similar role for $A_D$ as the Fourier kernel $e^{i \xi x}$ ($\xi \in \R$) for $A$. This is formally stated in the following result.

\begin{theorem}[{\cite[Theorem~1.1]{bib:k11} and~\cite[Theorem~1.3]{bib:kmr13}}]
\label{th:halfline}
The functions $F_\mu$ are $L^\infty$ eigenfunctions of $A_{(0, \infty)}$; the corresponding eigenvalues are $\psi(\mu^2)$. The operator $A_{(0, \infty)}$ takes a diagonal form under the integral transform with kernel $F_\mu(x)$. More precisely, let
\formula{
 \Pi f(\mu) & = \int_0^\infty f(x) F_\mu(x) dx
}
for $f \in L^2((0, \infty)) \cap L^1((0, \infty))$. Then $(\tfrac{2}{\pi})^{1/2} \Pi$ extends to a unitary mapping on $L^2((0, \infty))$, such that for $f \in L^2((0, \infty))$,
\formula{
 f \in \domain(A_{(0, \infty)}) & \iff (1 + \psi(\mu^2)) \Pi f(\mu) \in L^2((0, \infty)) ,
}
and if $f \in \domain(A_{(0, \infty)})$, then
\formula{
 \Pi (A_{(0, \infty)} f)(\mu) & = \psi(\mu^2) \Pi f(\mu) , & \Pi (T_D(t) f) & = e^{-t \psi(\mu^2)} \Pi f(\mu) .
}
\end{theorem}

In this article we only use the first part of the above result, namely, that $F_\mu$ are the $L^\infty((0, \infty))$ eigenfunctions of $A_{(0, \infty)}$. We remark that a similar eigenfunction expansion is available for $D = \R \setminus \{0\}$, see~\cite{bib:jk15,bib:k12a}, and there are no other known explicit expressions for the eigenfunctions of $A_D$ unless $D = \R$ or $\psi(\xi) = c \xi + \tilde{c}$.

\subsection{Estimates of $\thet_\mu$}
\label{sec:prethetmu}

Recall that according to~\eqref{eq:thetmu}, Definition~\ref{def:halfline} and~\cite[Proposition~4.16]{bib:k11},
\formula[eq:thetmualt]{
 \thet_\mu & = \frac{1}{\pi} \int_0^\infty \frac{\mu}{s^2 - \mu^2} \, \log \frac{\psi'(\mu^2) (\mu^2 - s^2)}{\psi(\mu^2) - \psi(s^2)} \, \D s \\
 & = \frac{1}{\pi} \int_0^1 \frac{1}{1 - z^2} \, \log \frac{\psi(\mu^2) - \psi(\mu^2 z^2)}{z^2 (\psi(\mu^2/z^2) - \psi(\mu^2))} \, \D z .
}
We remark that if $\psi$ is regularly varying at infinity with index $\alpha \in (0, 2]$, then, by dominated convergence,
\formula[eq:thetreg]{
\begin{aligned}
 \lim_{\mu \to \infty} \thet_\mu & = \frac{1}{\pi} \lim_{\mu \to \infty} \int_0^1 \frac{1}{1 - z^2} \, \log \frac{1 - \psi(\mu^2 z^2) / \psi(\mu^2)}{z^2 (\psi(\mu^2/z^2) / \psi(\mu^2) - 1)} \, \D z \\
 & = \frac{1}{\pi} \int_0^1 \frac{1}{1 - z^2} \, \log \frac{1 - z^\alpha}{z^2 (z^{-\alpha} - 1)} \, \D z = \frac{2 - \alpha}{\pi} \int_0^1 \frac{-\log z}{1 - z^2} \, \D z = \frac{(2 - \alpha) \pi}{8} \, ,
\end{aligned}
}
see~\cite[Example~6.1]{bib:k11}. By~\cite[Proposition~4.17]{bib:k11}, dominated convergence can be used to differentiate the right-hand side of~\eqref{eq:thetmualt} in $\mu > 0$ under the integral sign. This yields
\formula[eq:dthetmu]{
 \frac{d \thet_\mu}{d \mu} & = \frac{2}{\pi \mu} \int_0^1 \frac{1}{1 - z^2} \expr{\frac{\mu^2 \psi'(\mu^2) - \mu^2 z^2 \psi'(\mu^2 z^2)}{\psi(\mu^2) - \psi(\mu^2 z^2)} - \frac{(\mu^2/z^2) \psi'(\mu^2/z^2) - \mu^2 \psi'(\mu^2)}{\psi(\mu^2/z^2) - \psi(\mu^2)}} \D z
}
for all $\mu > 0$. In this section we prove two properties of $\thet_\mu$ that are needed in the remaining part of the article. First, we find estimates of $\thet_\mu$ that imply that the lower limits of $\thet_\mu$ as $\mu \to 0^+$ or $\mu \to \infty$ do not exceed $\tfrac{3 \pi}{8}$ (Lemma~\ref{lem:thetmulim}). Next, a simple estimate of $\tfrac{d}{d \mu} \thet_\mu$ is found (Lemma~\ref{lem:dthetmuest}).

By~\cite[Proposition~4.3]{bib:kmr13}, we have the following general estimate of $\thet_\mu$:
\formula{
 \expr{\inf_{\xi > 0} \frac{-\xi \psi''(\xi)}{\psi'(\xi)}} \frac{\pi}{4} & \le \thet_\mu \le \expr{\sup_{\xi > 0} \frac{-\xi \psi''(\xi)}{\psi'(\xi)}} \frac{\pi}{4}
}
for all $\mu > 0$. Furthermore, the supremum in the upper bound is always not greater than $2$. If $\psi$ is a Thorin--Bernstein function (see~\cite{bib:ssv10}), then one easily checks that the supremum is in fact not greater than $1$, and therefore $\thet_\mu \le \tfrac{\pi}{4}$. Below we find more refined bounds for $\thet_\mu$.

By~\cite[Proposition~4.5]{bib:kmr13},
\formula[eq:thetmupoint1]{
 \frac{1}{\pi} \expr{\arcsin^2 \sqrt{Q} + \arcsin^2 \sqrt{\frac{Q}{1 - P}} - \arcsin^2 \sqrt{\frac{P Q}{1 - P}}} & \le \thet_\mu \le \frac{\pi}{2} - \arcsin \sqrt{P}
}
with
\formula{
 P & = \frac{\mu^2 \psi'(\mu^2)}{\psi(\mu^2)} \, , & Q & = \frac{-\mu^2 \psi''(\mu^2)}{2 \psi'(\mu^2)}
}
(note that the factor $\tfrac{1}{\pi}$ is missing in the lower bound in the original statement). By the same argument as in the proof of the lower bound of~\cite[Proposition~4.5]{bib:kmr13} (using the lower bound for $\psi_\lambda(\lambda^2 \zeta^2)$ and the upper bound for $\psi_\lambda(\lambda^2 / \zeta^2)$), one easily shows that, with the same $P$ and $Q$,
\formula[eq:thetmupoint2]{
 \thet_\mu & \le \frac{\pi}{4} - \frac{1}{\pi} \expr{\arcsin^2 \sqrt{1 - Q} + \arcsin^2 \sqrt{1 - \frac{Q}{1 - P}} - \arcsin^2 \sqrt{1 - \frac{P Q}{1 - P}}} .
}
One can also verify that this bound is always at least as good as the upper bound of~\eqref{eq:thetmupoint1}, with equality when $P + Q = 1$.

The following technical result states that $P + Q \le 1$. This in fact follows indirectly from the proof of~\cite[Proposition~4.5]{bib:kmr13} (note that the right-hand side of~\eqref{eq:thetmupoint2} is not well-defined when $P + Q > 1$), but we choose to give a simple, direct argument.

\begin{lemma}
\label{lem:cbfest}
If $\psi$ is a non-constant complete Bernstein function, then
\formula{
 \frac{-\xi \psi''(\xi)}{\psi'(\xi)} & \le 2 - \frac{2 \xi \psi'(\xi)}{\psi(\xi)} \, .
}
\end{lemma}

\begin{proof}
The lemma is equivalent to the inequality
\formula{
 -\xi \psi(\xi) \psi''(\xi) & \le 2 \psi'(\xi) (\psi(\xi) - \xi \psi'(\xi)) .
}
Assuming $\psi$ has the representation~\eqref{eq:cbf}, we need to prove
\formula{
 & \xi \expr{c \xi + \tilde{c} + \frac{1}{\pi} \int_{(0, \infty)} \frac{\xi}{\xi + s} \, \frac{\mu(\D s)}{s}} \expr{\frac{1}{\pi} \int_{(0, \infty)} \frac{s}{(\xi + s)^3} \, \frac{\mu(\D s)}{s}} \\
 & \qquad \le \expr{c + \frac{1}{\pi} \int_{(0, \infty)} \frac{s}{(\xi + s)^2} \, \frac{\mu(\D s)}{s}} \expr{\tilde{c} + \frac{1}{\pi} \int_{(0, \infty)} \frac{\xi^2}{(\xi + s)^2} \, \frac{\mu(\D s)}{s}} .
}
This follows by simple integration from the following bounds: $0 \le c \tilde{c}$,
\formula{
 & \xi (c \xi + \tilde{c}) \, \frac{s}{(\xi + s)^3} \le c \, \frac{\xi^2}{(\xi + s)^2} + \tilde{c} \, \frac{s}{(\xi + s)^2} \, ,
}
and
\formula{
 & \xi \expr{\frac{\xi}{\xi + s_1} \, \frac{s_2}{(\xi + s_2)^3} + \frac{\xi}{\xi + s_2} \, \frac{s_1}{(\xi + s_1)^3}} \le \frac{s_1}{(\xi + s_1)^2} \, \frac{\xi^2}{(\xi + s_2)^2} + \frac{s_2}{(\xi + s_2)^2} \, \frac{\xi^2}{(\xi + s_1)^2} \, ;
}
the last two inequalities are easily proved by direct calculations.
\end{proof}

\begin{lemma}
\label{lem:increasing}
The left-hand side of~\eqref{eq:thetmupoint1} is decreasing in $P \in [0, 1 - Q]$. The right-hand side of~\eqref{eq:thetmupoint2} is increasing in $P \in [0, 1 - Q]$.
\end{lemma}

\begin{proof}
Let $P = 1 - \tfrac{Q}{s + Q} = \tfrac{s}{s + Q}$, $s \in [0, 1 - Q]$. Note that $P$ increases with increasing $s$, and the left-hand side of~\eqref{eq:thetmupoint1} is equal to
\formula{
 \tfrac{1}{\pi} (\arcsin^2 \sqrt{Q} + \arcsin^2 \sqrt{s + Q} - \arcsin^2 \sqrt{s}) .
}
Since $\arcsin^2 \sqrt{s}$ is convex, the above expression is increasing in $s$. In a similar way, with $P = 1 - \tfrac{Q}{1 - s} = \tfrac{1 - s - Q}{1 - s}$, $s \in [0, 1 - Q]$, the right-hand side of~\eqref{eq:thetmupoint2} is equal to
\formula{
 \tfrac{\pi}{4} - \tfrac{1}{\pi} (\arcsin^2 \sqrt{1 - Q} + \arcsin^2 \sqrt{s} - \arcsin^2 \sqrt{s + Q}) ,
}
which is again an increasing function of $s$, but now $P$ decreases with increasing $s$.
\end{proof}

Substituting $P = 0$, we obtain immediately the following elegant result.

\begin{corollary}
\label{cor:inversebound}
If $\psi$ is a non-constant complete Bernstein function such that $\psi(0) = 0$, then
\formula{
 \frac{2}{\pi} \arcsin^2 \sqrt{\frac{-\mu^2 \psi''(\mu^2)}{2 \psi'(\mu^2)}} & \le \thet_\mu \le \frac{\pi}{2} - \frac{2}{\pi} \arcsin^2 \sqrt{1 + \frac{\mu^2 \psi''(\mu^2)}{2 \psi'(\mu^2)}} \, .
}
\end{corollary}

\begin{lemma}
\label{lem:thetmulim}
If $\psi$ is a non-constant complete Bernstein function such that $\psi(0) = 0$, then
\formula{
 \liminf_{\mu \to 0^+} \thet_\mu \le \frac{3 \pi}{8} \, .
}
If $\psi$ is unbounded, then also
\formula{
 \liminf_{\mu \to \infty} \thet_\mu \le \frac{3 \pi}{8} \, .
}
\end{lemma}

\begin{proof}
Suppose that $\liminf_{\mu \to 0^+} \thet_\mu > \tfrac{3 \pi}{8}$. Then there are $\mu_0 > 0$ and $q \in (0, 1)$ such that $\thet_\mu \ge \tfrac{\pi}{2} - \tfrac{q \pi}{8}$ for $\mu \in (0, \mu_0)$. By Corollary~\ref{cor:inversebound},
\formula{
 \arcsin^2 \sqrt{1 + \frac{\mu^2 \psi''(\mu^2)}{2 \psi'(\mu^2)}} \le \frac{q \pi^2}{16}
}
for $\mu \in (0, \mu_0)$, and hence
\formula{
 \frac{-\mu^2 \psi''(\mu^2)}{\psi'(\mu^2)} \ge 2 - 2 \expr{\sin \frac{\pi \sqrt{q}}{4}}^2
}
for $\mu \in (0, \mu_0)$. If $\alpha$ denotes the right-hand side, then $\alpha > 1$. By integration (see~\cite[Lemma~2.2]{bib:jk15}), we have $\psi'(\mu^2) / \psi'(\mu_0^2) \ge (\mu_0^2 / \mu^2)^\alpha$ for all $\mu \in (0, \mu_0)$, which contradicts integrability of $\psi'$ at $0$. This proves the first statement of the lemma.

In a similar manner, if $\liminf_{\mu \to \infty} \thet_\mu > \tfrac{3 \pi}{8}$, then there are $\mu_0 > 0$ and $q \in (0, 1)$ such that $\thet_\mu \ge \tfrac{\pi}{2} - \tfrac{q \pi}{8}$ for $\mu \in (\mu_0, \infty)$. Again this implies
\formula{
 \frac{-\mu^2 \psi''(\mu^2)}{\psi'(\mu^2)} \ge 2 - 2 \expr{\sin \frac{\pi \sqrt{q}}{4}}^2
}
for $\mu \in (\mu_0, \infty)$. If $\alpha$ denotes the right-hand side, then $\alpha > 1$, and by integration, $\psi'(\mu^2) / \psi'(\mu_0^2) \le (\mu_0^2 / \mu^2)^\alpha$ for all $\mu \in (\mu_0, \infty)$. This implies integrability of $\psi'$ at $\infty$.
\end{proof}

We conjecture that the above lemma holds with $\tfrac{3 \pi}{8}$ replaced with $\tfrac{\pi}{4}$. An example of a complete Bernstein function $\psi$ for which the set of partial limits of $\thet_\mu$ as $\mu \to 0^+$ is equal to $[0, \tfrac{\pi}{2}]$ is given in~\cite[Section~7.5]{bib:kmr13}.

\begin{lemma}
\label{lem:dthetmuest}
If $\psi$ is a non-constant complete Bernstein function such that $\psi(0) = 0$, then for all $\mu > 0$,
\formula{
 \abs{\frac{d \thet_\mu}{d \mu}} & < \frac{3}{\mu} \, .
}
\end{lemma}

\begin{proof}
By~\eqref{eq:dthetmu} and the Cauchy's mean value theorem, for some $\xi_z \in (\mu^2 z^2, \mu^2)$ and $\xi_{1/z} \in (\mu^2, \mu^2 / z^2)$ (where $z \in (0, 1)$),
\formula{
 \frac{d \thet_\mu}{d \mu} & = \frac{2}{\pi \mu} \int_0^1 \frac{1}{1 - z^2} \expr{\frac{\xi_z \psi''(\xi_z) + \psi'(\xi_z)}{\psi'(\xi_z)} - \frac{\xi_{1/z} \psi''(\xi_{1/z}) + \psi'(\xi_{1/z})}{\psi'(\xi_{1/z})}} \D z \\
 & = \frac{2}{\pi \mu} \int_0^1 \frac{1}{1 - z^2} \expr{\frac{\xi_z \psi''(\xi_z)}{\psi'(\xi_z)} - \frac{\xi_{1/z} \psi''(\xi_{1/z})}{\psi'(\xi_{1/z})}} \D z .
}
By~\eqref{eq:cbf}, $0 \le -\xi \psi''(\xi) \le 2 \psi'(\xi)$ and $0 \le \xi^2 \psi^{(3)}(\xi) \le 6 \psi'(\xi)$. Hence,
\formula{
 \xi \, \frac{d}{d \xi}\!\expr{\frac{\xi \psi''(\xi)}{\psi'(\xi)}} & = \frac{\xi^2 \psi^{(3)}(\xi)}{\psi'(\xi)} - \frac{-\xi \psi''(\xi)}{\psi'(\xi)} - \frac{(-\xi \psi''(\xi))^2}{(\psi'(\xi))^2} \in [-6, 6].
}
Furthermore, $\xi_z \psi''(\xi_z) / \psi'(\xi_z) - \xi_{1/z} \psi''(\xi_{1/z}) / \psi'(\xi_{1/z}) \in [-2, 2]$. It follows that
\formula{
 \abs{\frac{d \thet_\mu}{d \mu}} & \le \frac{2}{\pi \mu} \int_0^1 \frac{1}{1 - z^2} \min\expr{2, \int_{\xi_{1/z}}^{\xi_z} \abs{\frac{d}{d r}\!\expr{\frac{r \psi''(r)}{\psi'(r)}}} dr} \D z \\
 & \le \frac{2}{\pi \mu} \int_0^1 \frac{1}{1 - z^2} \, \min \expr{2, 6 \log \frac{\xi_{1/z}}{\xi_z}} \D z .
}
Recall that $\xi_{1/z} / \xi_z \le z^{-4}$. Hence,
\formula{
 \abs{\frac{d \thet_\mu}{d \mu}} & \le \frac{2}{\pi \mu} \int_0^1 \frac{\min(2, -24 \log z)}{1 - z^2} \, \D z .
}
Since $-\log z \le \tfrac{1}{z} - 1$, we have
\formula{
 \abs{\frac{d \thet_\mu}{d \mu}} & \le \frac{2}{\pi \mu} \int_0^1 \frac{\min(2, 24 (1/z - 1))}{1 - z^2} \, \D z = \frac{100 \log 5 - 48 \log 24}{\pi \mu} < \frac{3}{\mu} \, .
\qedhere
}
\end{proof}

We conjecture that in fact $-\tfrac{1}{\mu} < \tfrac{d}{d \mu} \thet_\mu \le \tfrac{1}{2 \mu}$. 
We close this section with the following simple example.

\begin{example}
\label{ex:sum}
Let $\psi(\xi) = \xi^{\alpha/2} + \xi^{\beta/2}$, where $0 < \beta < \alpha \le 2$. By a short calculation,
\formula{
 \frac{\psi(\mu^2) - \psi(\mu^2 z^2)}{z^2 (\psi(\mu^2/z^2) - \psi(\mu^2))} & = \frac{1}{z^{2 - \alpha}} \expr{\frac{1 - z^\alpha}{1 - z^\beta} + \frac{1}{\mu^{\alpha - \beta}}} \expr{\frac{1 - z^\alpha}{1 - z^\beta} + \frac{z^{\alpha - \beta}}{\mu^{\alpha - \beta}}}^{-1} .
}
If we denote $w = (1 - z^\alpha) / (1 - z^\beta)$, then
\formula{
 \frac{\psi(\mu^2) - \psi(\mu^2 z^2)}{z^2 (\psi(\mu^2/z^2) - \psi(\mu^2))} & = \frac{1}{z^{2 - \alpha}} \, \frac{\mu^{\alpha - \beta} w + 1}{\mu^{\alpha - \beta} w + z^{\alpha - \beta}} = \frac{1}{z^{2 - \alpha}} \expr{1 + \frac{1 - z^{\alpha - \beta}}{\mu^{\alpha - \beta} w + z^{\alpha - \beta}}} .
}
As in the last equality of~\eqref{eq:thetreg}, we obtain
\formula{
 \thet_\mu & = \frac{(2 - \alpha) \pi}{8} + \frac{1}{\pi} \int_0^1 \frac{1}{1 - z^2} \log \expr{1 + \frac{1 - z^{\alpha - \beta}}{\mu^{\alpha - \beta} w + z^{\alpha - \beta}}} dz .
}
Clearly, the integrand is nonnegative. Since $\log(1 + s) \le s$, $z^{\alpha - \beta} \ge z^2$ and $w \ge 1$,
\formula{
 \int_0^1 \frac{1}{1 - z^2} \log \expr{1 + \frac{1 - z^{\alpha - \beta}}{\mu^{\alpha - \beta} w + 1}} dz & \le \int_0^1 \frac{1}{1 - z^2} \, \frac{1 - z^{\alpha - \beta}}{\mu^{\alpha - \beta} w + z^{\alpha - \beta}} \, dz \\
 & \le \int_0^1 \frac{1}{1 - z^2} \, \frac{1 - z^{2}}{\mu^{\alpha - \beta}} \, dz = \frac{1}{\mu^{\alpha - \beta}} \, .
}
Therefore,
\formula{
 \frac{(2 - \alpha) \pi}{8} \le \thet_\mu & \le \frac{(2 - \alpha) \pi}{8} + \frac{1}{\pi \mu^{\alpha - \beta}} \, .
}
\end{example}

\subsection{Estimates of $F_\mu(x)$}
\label{sec:pregmu}

In the remaining part of the article we will need the following simple estimate of $\laplace F_\mu$ and a more refined estimate of $G_\mu$.

\begin{lemma}
\label{lem:laplaceflambdaest}
If $\psi$ is a non-constant complete Bernstein function such that $\psi(0) = 0$, then for all $\mu > 0$ and $\xi$ such that $\real \xi > 0$,
\formula{
 |\laplace F_\mu(\xi)| & \le 2 \sqrt{2} \, \frac{\mu}{|\mu^2 + \xi^2|} \sqrt{\frac{\psi'(\mu^2) (\mu^2 - |\xi|^2)}{\psi(\mu^2) - \psi(|\xi|^2)}} \, .
}
\end{lemma}

\begin{proof}
Recall that $(\mu^2 + \xi^2) \laplace F_\mu(\xi) = \mu (\psi_\mu(\mu^2))^{-1/2} \psi^\dagger_\mu(\xi)$ is a complete Bernstein function of $\xi$, and hence by~\cite[Proposition~2.21(c)]{bib:k11} and~\cite[Corollary~5.1]{bib:kmr13},
\formula{
 |\mu^2 + \xi^2| \, |\laplace F_\mu(\xi)| & \le \sqrt{2} \, (\mu^2 + |\xi|^2) \laplace F_\mu(|\xi|) \le 2 \sqrt{2} \mu \sqrt{\frac{\psi'(\mu^2) (\mu^2 - |\xi|^2)}{\psi(\mu^2) - \psi(|\xi|^2)}}
}
for all $\xi$ such that $\real \xi > 0$.
\end{proof}

\begin{lemma}
\label{lem:glambdaest}
If $\psi$ is a non-constant complete Bernstein function such that $\psi(0) = 0$, then for all $\mu, x > 0$ such that $\mu x \ne 1$,
\formula{
 G_\mu(x) & \le \frac{1}{\pi x} \, \frac{\psi(1 / x^2)}{\psi(\mu^2)} \, \sqrt{\frac{\psi'(\mu^2)}{\psi(\mu^2)}} \, \frac{1 - \psi(\mu^2) / (\mu^2 x^2 \psi(1/x^2))}{1 - \psi(1/x^2) / \psi(\mu^2)} \, .
}
In particular, if $\psi$ is unbounded, then
\formula{
 \limsup_{\mu \to \infty} (\mu \psi(\mu^2) G_\mu(x)) & \le \frac{\psi(1 / x^2)}{\pi x} \, .
}
\end{lemma}

\begin{proof}
Recall that $\psi^\dagger_\mu(\xi) \ge \psi^\dagger_\mu(0) = \psi_\mu(0) = 1$. Hence,
\formula{
 \gamma_\mu(d \xi) & \le \frac{1}{\pi \mu \sqrt{\psi_\mu(\mu^2)}} \lim_{\eps \to 0^+} \imag \psi_\mu(-e^{-i \eps} \xi^2) d\xi .
}
After a substitution $\xi = \sqrt{s}$ it follows that
\formula{
 G_\mu(x) = \int_0^\infty e^{-\xi x} \gamma_\mu(d \xi) & \le \frac{1}{2 \pi \mu \sqrt{\psi_\mu(\mu^2)}} \lim_{\eps \to 0^+} \int_0^\infty \sqrt{s} \, e^{-x \sqrt{s}} \, \frac{\imag \psi_\mu(-e^{-i \eps} s) ds}{s} \, .
}
Since $x \sqrt{s} \, e^{-x \sqrt{s}} \le 2 / (1 + x^2 s)$, we have
\formula{
 G_\mu(x) & \le \frac{1}{\pi \mu x \sqrt{\psi_\mu(\mu^2)}} \lim_{\eps \to 0^+} \int_0^\infty \frac{1}{1 + x^2 s} \, \frac{\imag \psi_\mu(-e^{-i \eps} s) ds}{s} \le \frac{\psi_\mu(1/x^2) - 1}{\pi \mu x \sqrt{\psi_\mu(\mu^2)}} \, ;
}
for the last inequality note that the integral converges to the integral term in the representation~\eqref{eq:cbf} for the complete Bernstein function $\psi_\mu$, and we have $\psi_\mu(0) = 1$ (therefore the inequality becomes equality if $\psi_\mu$ contains no linear term, that is, if $\psi$ is unbounded). To prove the first statement, it remains to use the definition of $\psi_\mu$. The other statement of the lemma follows from the first one by the inequality $\xi \psi'(\xi) \le \psi(\xi)$.
\end{proof}

%
%

\section{Proofs}
\label{sec:proofs}

Throughout this section we implicitly assume that $\psi$ is a non-constant complete Bernstein function such that $\psi(0) = 0$, that is, $\tilde{c} = 0$ in the representation~\eqref{eq:cbf} for $\psi$. By $c$ and $\nu$ we denote the constant and the measure in the representation~\eqref{eq:cbf} for $\psi$. Finally, we let $D = (-a, a)$ for some $a > 0$.

\subsection{Pointwise estimates for the operator $A$}
\label{sec:prpoint}

Recall that $A = \psi(-d^2 / dx^2)$, and for $f \in C_c^\infty(\R)$ we have, as in~\eqref{eq:gen},
\formula[eq:gen2]{
\begin{aligned}
 A f(x) & = -c f''(x) + \pvint_{-\infty}^\infty (f(x) - f(y)) \nu(x - y) dy \\
 & = -c f''(x) + \int_0^\infty (2 f(x) - f(x + z) - f(x - z)) \nu(z) dz .
\end{aligned}
}
We denote the right-hand side by $\A f(x)$ (with a calligraphic letter $\A$) whenever the integral converges and, if $c > 0$, $f''$ is well-defined. The following estimates of $\A f(x)$ are proved in~\cite{bib:kkm13} in the special case $\psi(\xi) = (\xi + 1)^{1/2} - 1$, but their proofs rely only on the symmetry, unimodality and positivity of the kernel function $\nu$. Note that in~\cite{bib:kkm13} the notation $\A_0$ is used for $\A$.

\begin{lemma}[{\cite[Proposition~4.1]{bib:kkm13}}]
\label{lem:genest1}
Let $x \in \R$, $b > 0$, and let $g$ have an absolutely continuous derivative in $(x - b, x + b)$. Then
\formula{
 \abs{\A g(x)} & \le c |g''(x)|  + \expr{\sup_{y \in (x - b, x + b)} |g''(y)|} \int_0^b z^2 \nu(z) dz \\
 & \hspace*{3em} + \int_{\R \setminus (x - b, x + b)} (|g(x)| + |g(y)|) \nu(y - x) dy .
}
\end{lemma}


As in~\cite{bib:kkm13,bib:kkms10,bib:k12}, for $b > 0$ we define an auxiliary function:
\formula[eq:q]{
 q(x) & = \begin{cases}
 0 & \text{for } x \in (-\infty, -b] , \\
 (1/2) (x/b + 1)^2 & \text{for } x \in [-b, 0] , \\
 1 - (1/2) (x/b - 1)^2 & \text{for } x \in [0, b] , \\
 1 & \text{for } x \in [b, \infty) .
 \end{cases}
}
Note that $q$ is $C^1$, $q'$ is absolutely continuous, $0 \le q''(x) \le 1 / b^2$ (for $x \in \R \setminus \{-b, 0, b\}$), the distributional derivative $q^{(3)}$ is a finite signed measure, and $q(x) + q(-x) = 1$.

\begin{lemma}[{\cite[Proposition~4.2]{bib:kkm13}}]
\label{lem:genest}
Let $b > 0$, let $f \in L^1(\R)$, and suppose that the second derivative $f''(x)$ exists for $x \in [-b, b]$ and it is continuous in $[-b, b]$. Define
\formula{
 M_{-1} & = \int_0^\infty |f(x)| dx , & M_0 & = \sup_{x \in [-b, b]} |f(x)| , \\
 M_1 & = \sup_{x \in [-b, b]} |f'(x)| , & M_2 & = \sup_{x \in [-b, b]} |f''(x)| ,
}
Let $q(x)$ be given by~\eqref{eq:q}, and define $g(x) = q(x) f(x)$. For $x \in (-\infty, 0)$, we have
\formula{
 \abs{\A g(x)} & \le C(b, \psi) (M_{-1} + M_0 + M_1 + M_2) .
}
More precisely, for $x \in (-\infty, -b]$ we have
\formula{
 \abs{\A g(x)} & \le \frac{M_0}{2 b^2} \int_0^{2b} z^2 \nu(z) dz + \nu(2b) M_{-1}
}
and for $x \in (-b, 0)$,
\formula{
 \abs{\A g(x)} & \le M_2 c + \expr{\frac{M_0}{b^2} + \frac{2 M_1}{b} + M_2} \int_0^b z^2 \nu(z) dz + 2 M_0 \int_b^\infty \nu(z) dz + \nu(b) M_{-1} .
}
\end{lemma}


\subsection{Approximate eigenfunctions}
\label{sec:prconstr}

Recall that $D = (-a, a)$. Following~\cite{bib:kkm13,bib:kkms10,bib:k12}, for $n \ge 1$, let $\tilde{\mu}_n$ be the largest solution of
\formula[eq:tildemu]{
 a \tilde{\mu}_n + \thet_{\tilde{\mu}_n} = \tfrac{n \pi}{2} ,
}
with $\thet_\mu$ defined in~\eqref{eq:thetmu}; this agrees with the definition of $\mu_n$ in~\eqref{eq:mun}, but we choose to use the notation $\tilde{\mu}_n$, so that all approximations are clearly distinguished from true values by the presence of a tilde. Although we are interested in large $n$ only, note that by Lemma~\ref{lem:thetmulim}, the equation $a \mu + \thet_\mu = \tfrac{n \pi}{2}$ has a solution for all $n \ge 1$, and every such solution satisfies
\formula{
 \tfrac{(n - 1) \pi}{2 a} \le \tilde{\mu}_n & \le \tfrac{n \pi}{2 a} .
}
We remark that~\eqref{eq:tildemu} may fail to have a unique solution for $n = 1$ (for example, when $a = 1$ and $\psi(\xi) = \xi / (10^4 + \xi) + \xi / 10^7$). Nevertheless, if $n \ge 3$ and $\mu \ge \tfrac{(n - 1) \pi}{2 a} = \tfrac{\pi}{a}$, then, by Lemma~\ref{lem:dthetmuest},
\formula{
 \tfrac{d}{d \mu} \expr{a \mu + \thet_\mu} & > a - \tfrac{3}{\mu} \ge a - \tfrac{3 a}{\pi} > 0 ,
}
and so the solution $\tilde{\mu}_n$ is in fact unique.

We let
\formula{
 \tilde{\lambda}_n & = \psi(\tilde{\mu}_n^2) .
}
In order to show that $\tilde{\lambda}_n$ is close to some eigenvalue of $A_D$, we construct an approximate eigenfunction $\tilde{\ph}_n$ of $A_D$, using the eigenfunctions $F_{\tilde{\mu}_n}(a - x)$, $F_{\tilde{\mu}_n}(a + x)$ for the one-sided problems corresponding to $A_{(-\infty, a)}$ and $A_{(-a, \infty)}$. As in~\cite{bib:kkm13,bib:kkms10,bib:k12}, we define
\formula[eq:phitilde]{
 \tilde{\ph}_n(x) & = q(-x) F_{\tilde{\mu}_n}(a + x) - (-1)^n q(x) F_{\tilde{\mu}_n}(a - x) ,
}
with the auxiliary function $q$ defined by~\eqref{eq:q}. Here $x \in \R$, but we have $\tilde{\ph}_n(x) = 0$ for $x \notin D$, so that $\tilde{\ph}_n$ is equal to zero in the complement of $D$. Clearly, $\tilde{\ph}_n$ is continuously differentiable in $D$, $\tilde{\ph}_n'$ is absolutely continuous in $D$, $\tilde{\ph}_n''$ exists in $D \setminus \{-b, b\}$, and $\tilde{\ph}_n''$ is locally bounded in $D$. Note that $\tilde{\lambda}_n$ depends on $a$ and $n$, while $\tilde{\ph}_n(x)$ depends also on $b$. We could fix $b$ in order to optimise the constants (in many cases $b = \tfrac{1}{3} a$ seems to be a reasonable choice), but since we do not track the exact value of the constants, we will simply indicate their dependence on $b$. Note also that $\tilde{\ph}_n$ is not normed in $L^2(D)$, its norm is approximately equal to $\sqrt{a}$ (see Lemma~\ref{lem:norm0}).

The notation introduced above is kept throughout this section.

The following result is intuitively clear, although its formal proof is rather long and technical.

\begin{lemma}[{see~\cite[Lemma~4.1]{bib:kkm13}}]
\label{lem:domain}
We have $\tilde{\ph}_n \in \domain(A_D)$ and $A_D \tilde{\ph}_n(x) = \A \tilde{\ph}_n(x)$ for almost all $x \in D$.
\end{lemma}

\begin{proof}
For brevity, in this proof we write $\tilde{\mu} = \tilde{\mu}_n$ and $\tilde{\ph} = \tilde{\ph}_n$. The domain of $A_D$ is described in Definition~\ref{def:h}: we need to prove that $\tilde{\ph} \in \domain(\form)$ and that $\tscalar{\tilde{\ph}, \A g} = \tscalar{\A \tilde{\ph}, g}$ for all $g \in C_c^\infty(D)$. We first verify the latter condition.

Note that $\A \tilde{\ph}(x)$ is well-defined for all $x \in D \setminus \{-b, b\}$, since $\tilde{\ph}$ is smooth in $D \setminus \{-b, b\}$ and bounded on $\R$. Let $g \in C_c^\infty(D)$. Since $\tilde{\ph}'$ is absolutely continuous in $(-a, a)$, integration by parts gives
\formula{
 \int_{-a}^a (-c \tilde{\ph}''(x)) g(x) dx = \int_{-a}^a \tilde{\ph}(x) (-c g''(x)) dx .
}
Furthermore, by the definition of $\A$ (see~\eqref{eq:gen2}),
\formula{
 \int_{-a}^a \A \tilde{\ph}(x) g(x) dx - \int_{-a}^a \tilde{\ph}(x) \A g(x) dx & = \\
 & \hspace*{-16em} \int_{-a}^a \expr{\int_0^\infty (g(x + z) \tilde{\ph}(x) + g(x - z) \tilde{\ph}(x) - g(x) \tilde{\ph}(x + z) - g(x) \tilde{\ph}(x - z)) \nu(z) dz} dx .
}
We claim that the double integral exists. Then, by Fubini, it is equal to $0$, and so $\tscalar{\tilde{\ph}, \A g} = \tscalar{\A \tilde{\ph}, g}$, as desired.

Denote the integrand by $I(x, z) \nu(z)$, and let $\eps = \tfrac{1}{3} \dist(\supp g, \R \setminus D)$, so that $\supp g \sub (-a + 3 \eps, a - 3 \eps)$. When $z \ge \eps$, then $|I(x, z)| \le 4 \|\tilde{\ph}\|_{L^\infty(\R)} \|g\|_{L^\infty(\R)}$. Suppose that $z \in (0, \eps)$. If $x \notin (-a + 2 \eps, a - 2 \eps)$, then $I(x, z) = 0$. Otherwise, by first-order Taylor's expansion of $I(x, z)$ around $z = 0$ (note that $I(x, 0) = \tfrac{\partial}{\partial z} I(x, 0) = 0$) with the remainder in the integral form, we obtain that
\formula{
 |I(x, z)| & \le \int_0^z (z - s) \, \frac{\partial^2}{\partial s^2} I(x, s) ds \\
 & \le z^2 (\|\tilde{\ph}\|_{L^\infty(\R)} \|g''\|_{L^\infty(\R)} + \|\tilde{\ph}''\|_{L^\infty((-a + \eps, a - \eps))} \|g\|_{L^\infty(\R)})
}
(recall that $\tilde{\ph}''$ is bounded in $(-a + \eps, a - \eps)$). We conclude that $|I(x, z) \nu(z)| \le C_1(\tilde{\ph}, g) \min(1, z^2) \nu(z)$, which implies joint integrability of $I(x, z) \nu(z)$. Our claim is proved.

It remains to verify that $\tilde{\ph} \in \domain(\form)$, that is, $(1 + \psi(\xi^2)) |\fourier \tilde{\ph}(\xi)|^2$ is integrable. Let $f(x) = q(a - x) F_{\tilde{\mu}}(x)$, so that $\tilde{\ph}(x) = f(a + x) - (-1)^n f(a - x)$ (see~\eqref{eq:phitilde}). It suffices to prove integrability of $(1 + \psi(\xi^2)) |\fourier f(\xi)|^2$.

Fix $\eps > 0$ and let $\tilde{q}(x) = q(a - x) e^{\eps x}$. Since the distributional derivatives $q$, $q'$ and $q''$ are integrable functions, and the third distributional derivative of $q(x)$ is a finite signed measure on $\R$, the function $\tilde{q}(x)$ has the same property. Therefore, $\fourier q(\xi)$ and $\fourier q^{(3)}(\xi) = -i \xi^3 \fourier q(\xi)$ are bounded functions, and so $|\fourier \tilde{q}(\xi)| \le C_2(\eps, a, b) / (1 + |\xi|)^3$. The Fourier transform of $e^{-\eps x} F_{\tilde{\mu}}(x)$ is equal to $\laplace F_{\tilde{\mu}}(\eps + i \xi)$, and the Fourier transform of $f(x) = q(a - x) F_{\tilde{\mu}}(x) = \tilde{q}(x) e^{-\eps x} F_{\tilde{\mu}}(x)$ is given by the convolution
\formula{
 \fourier f(\xi) & = \frac{1}{2 \pi} \int_{-\infty}^\infty \fourier \tilde{q}(\xi - s) \laplace F_{\tilde{\mu}}(\eps + i s) ds .
}
Suppose that $\xi > 0$. To estimate $|\fourier f(\xi)|$, we write
\formula[eq:fourierdecomposition]{
 \fourier f(\xi) & = \frac{1}{2 \pi} \int_{\xi/2}^\infty \fourier \tilde{q}(\xi - s) \laplace F_{\tilde{\mu}}(\eps + i s) ds + \frac{1}{2 \pi} \int_{\xi/2}^\infty \fourier \tilde{q}(s) \laplace F_{\tilde{\mu}}(\eps + i (\xi - s)) ds .
}
By Lemma~\ref{lem:laplaceflambdaest}, we have
\formula{
 |\laplace F_{\tilde{\mu}}(\eps + i s)| & \le 2 \sqrt{2} \, \frac{\tilde{\mu}}{|\tilde{\mu}^2 + (\eps + i s)^2|} \, \sqrt{\frac{\psi'(\tilde{\mu}^2) (\tilde{\mu}^2 - |\eps + i s|^2)}{\psi(\tilde{\mu}^2) - \psi(|\eps + i s|^2)}} \\
 & \le C_3(\eps, \tilde{\mu}, \psi) \, \frac{1}{1 + s} \, \sqrt{\frac{1}{1 + \psi(s^2)}}
}
(for the second inequality observe that the expression under the square root is bounded by a constant when $s \le 2 \tilde{\mu}$ and by $\psi'(\tilde{\mu}^2) (1 + s^2) / (\psi(s^2) - \psi(\tilde{\mu}^2))$ when $s > 2 \tilde{\mu}$). The right-hand side decreases with $s > 0$. Hence,
\formula{
 \hspace*{2em} & \hspace*{-2em} \abs{\int_{\xi/2}^\infty \fourier \tilde{q}(\xi - s) \laplace F_{\tilde{\mu}}(\eps + i s) ds} \le \frac{C_3(\eps, \tilde{\mu}, \psi)}{(1 + \xi / 2) (1 + \psi(\xi^2 / 4))^{1/2}} \, \int_{\xi/2}^\infty |\fourier \tilde{q}(\xi - s)| ds \\
 & \le \frac{C_3(\eps, \tilde{\mu}, \psi) C_2(\eps, a, b)}{(1 + \xi / 2) (1 + \psi(\xi^2 / 4))^{1/2}} \int_{\xi/2}^\infty \frac{1}{(1 + |\xi - s|)^3} \, ds \le \frac{8 C_3(\eps, \tilde{\mu}, \psi) C_2(\eps, a, b)}{(1 + \xi) (1 + \psi(\xi^2))^{1/2}} \, ;
}
in the last inequality we used the fact that $4 \psi(\xi^2 / 4) \ge \psi(\xi^2)$ and that the integral is bounded by $1$. The estimate of the other integral in~\eqref{eq:fourierdecomposition} is simpler: $|\laplace F_{\tilde{\mu}}(\eps + i s)| \le C_4(\eps, \tilde{\mu})$ for all $s \in \R$, and hence
\formula{
 \abs{\int_{\xi/2}^\infty \fourier \tilde{q}(s) \laplace F_{\tilde{\mu}}(\eps + i (\xi - s)) ds} & \le C_4(\eps, \tilde{\mu}) \int_{\xi/2}^\infty |\fourier \tilde{q}(s)| ds \le \frac{C_4(\eps, \tilde{\mu}) C_2(\eps, a, b)}{2 (1 + \xi/2)^2} \, .
}
Therefore, for $\xi > 0$,
\formula{
 |\fourier f(\xi)| & \le C_5(\eps, a, b, \tilde{\mu}) \expr{\frac{1}{(1 + |\xi|) (1 + \psi(\xi^2))^{1/2}} + \frac{1}{(1 + |\xi|)^2}} .
}
Since $\fourier f(-\xi) = \overline{\fourier f(\xi)}$, the above estimate extends to all $\xi \in \R$. We conclude that for all $\xi \in \R$,
\formula{
 (1 + \psi(\xi^2)) |\fourier f(\xi)|^2 & \le 2 (C_5(\eps, a, b, \tilde{\mu}))^2 \expr{\frac{1}{(1 + |\xi|)^2} + \frac{1 + \psi(\xi^2)}{(1 + |\xi|)^4}} ,
}
and the right-hand side is integrable because $(1 + |\xi|)^{-2} (1 + \psi(\xi^2))$ is bounded.
\end{proof}

\subsection{Estimates for approximate eigenfunctions}
\label{sec:prest}

Following~\cite{bib:kkm13}, we introduce the following notation:
\formula{
\begin{aligned}
 \nu_0(x) & = c + \int_0^x z^2 \nu(z) dz , \qquad & \nu_\infty(x) & = \int_x^\infty \nu(z) dz , \\
 I_\mu & = \int_0^\infty G_\mu(x) dx , \qquad & G_{\mu,b}(x) & = G_\mu(x - b) + G_\mu(x + b) .
\end{aligned}
}
We recall two fundamental estimates, which were proved in~\cite{bib:kkm13} for $\psi(\xi) = (\xi + 1)^{1/2} - 1$, but their proofs work for general non-constant complete Bernstein functions $\psi$ such that $\psi(0) = 0$. One minor change is required in the proof of Lemma~\ref{lem:approx0}: an extra term $M_2 c$ appears when Lemma~\ref{lem:genest} is applied (as compared to the application of~\cite[Proposition~4.2]{bib:kkm13} in the proof of~\cite[Lemma~4.2]{bib:kkm13}). This extra term is absorbed into $M_2 \nu_0(b)$. Also, note two typos in the first displayed formula in the original statement of~\cite[Lemma~4.2]{bib:kkm13}: the norm in the left-hand side should not be squared, and the term $\tilde{\lambda}_n G_{\tilde{\mu}_n,b}(a)$ is missing in the right-hand side. (These typos did not appear in the the other displayed formula in the original statement, which was the one used later in the proof of the main result.)

\begin{lemma}[{\cite[Lemma~4.2]{bib:kkm13}}]
\label{lem:approx0}
We have
\formula{
 \| A_D \tilde{\ph}_n - \tilde{\lambda}_n \tilde{\ph}_n \|_{L^2(D)} & \le C(a, b, \psi) ((1 + \tilde{\lambda}_n) G_{\tilde{\mu}_n,b}(a) - G_{\tilde{\mu}_n,b}'(a) + G_{\tilde{\mu}_n,b}''(a) + I_{\tilde{\mu}_n} + \tfrac{1}{\tilde{\mu}_n}) .
}
More precisely, we have
\formula{
\begin{aligned}
 \| A_D \tilde{\ph}_n - \tilde{\lambda}_n \tilde{\ph}_n \|_{L^2(D)}^2 & \le 2 (a - b) \expr{\frac{G_{\tilde{\mu}_n,b}(a) \nu_0(2b)}{2 b^2} + \nu(2b) I_{\tilde{\mu}_n} + \frac{2 \nu(a)}{\tilde{\mu}_n}}^2 \\
 & \hspace*{-6em} + 2 b \Biggl( \frac{(G_{\tilde{\mu}_n,b}(a) - 2 b G_{\tilde{\mu}_n,b}'(a) + b^2 G_{\tilde{\mu}_n,b}''(a)) \nu_0(b)}{b^2} \\
 & \hspace*{-1em} + 2 G_{\tilde{\mu}_n,b}(a) \nu_\infty(b) + \nu(b) I_{\tilde{\mu}_n} + \frac{\tilde{\lambda}_n G_{\tilde{\mu}_n,b}(a)}{2} + \frac{2 \nu(a)}{\tilde{\mu}_n} \Biggr)^2 .
\end{aligned}
}
\end{lemma}

\begin{lemma}[{\cite[Lemma~4.3]{bib:kkm13}}]
\label{lem:norm0}
We have
\formula{
 \bigl| \|\tilde{\ph}_n\|_{L^2(D)}^2 - a \bigr| & \le 8 (I_{\tilde{\mu}_n} + 1 / \tilde{\mu}_n) .
}
More precisely,
\formula{
 a - \frac{\sin(\thet_{\tilde{\mu}_n})}{\tilde{\mu}_n} - 4 I_{\tilde{\mu}_n} & \le \|\tilde{\ph}_n\|_{L^2(D)}^2 \le a + \frac{\sin(\thet_{\tilde{\mu}_n})}{\tilde{\mu}_n} + 4 I_{\tilde{\mu}_n} (1 + \sin \thet_{\tilde{\mu}_n}) .
}
\end{lemma}

\begin{lemma}
\label{lem:approxnorm0}
If $\psi$ is unbounded, then for $n \ge 2$,
\formula{
 \| A_D \tilde{\ph}_n - \tilde{\lambda}_n \tilde{\ph}_n \|_{L^2(D)} & \le \frac{C(a, b, \psi)}{n}
}
and
\formula{
 a - \frac{20 a}{n \pi} & \le \|\tilde{\ph}_n\|_{L^2(D)}^2 \le a + \frac{36 a}{n \pi} \, .
}
\end{lemma}

\begin{proof}
By~\cite[Lemma~4.21]{bib:k11},
\formula[eq:lgmu0]{
 I_\mu & = \laplace G_\mu(0) = \frac{\cos \thet_\mu}{\mu} - \laplace F_\mu(0^+) = \frac{\cos \thet_\mu}{\mu} - \sqrt{\frac{\psi'(\mu^2)}{\psi(\mu^2)}} \le \frac{1}{\mu} \, .
}
Furthermore, by complete monotonicity,
\formula{
 I_\mu & \ge \int_0^x G_\mu(z) dz \ge \int_0^x (G_\mu(x) - G_\mu'(x) (x - z) + \tfrac{1}{2} G_\mu''(x) (x - z)^2) dz \\
 & = x G_\mu(x) - \tfrac{1}{2} x^2 G_\mu'(x) + \tfrac{1}{6} x^3 G_\mu''(x) ,
}
so that
\formula{
 G_\mu(x) & \le \frac{1}{\mu x} \, , & G_\mu'(x) & \le \frac{2}{\mu x^2} \, , & G_\mu''(x) & \le \frac{6}{\mu x^3} \, .
}
By Lemma~\ref{lem:glambdaest}, for $\mu \ge \tilde{\mu}_2$,
\formula{
 \psi(\mu^2) G_\mu(x) & \le \frac{C(\psi, x)}{\mu} \, .
}
Finally, $\tilde{\mu}_n \ge \tfrac{(n - 1) \pi}{2 a} \ge \tfrac{n \pi}{4 a}$ for $n \ge 2$. The result follows from Lemmas~\ref{lem:approx0} and~\ref{lem:norm0}.
\end{proof}

\subsection{Distance to nearest eigenvalue}
\label{sec:prlambda}

Let $\sigma(A_D)$ denote the spectrum of $A_D$. Recall that the spectrum of $A_D$ is purely discrete (see Subsection~\ref{sec:prelambda}), and the eigenvalues of $A_D$ are denoted by $\lambda_n$. The following result was given in~\cite{bib:kkm13} for $\psi(\xi) = (\xi + 1)^{1/2} - 1$ only, but the proof extends to arbitrary self-adjoint operators $A_D$ that preserve the spaces of even and odd functions.

\begin{lemma}[{\cite[Proposition~4.2]{bib:kkm13}}]
\label{lem:dist0}
We have
\formula[eq:dist0]{
 \dist(\tilde{\lambda}_n, \sigma(A_D)) & \le \frac{\| A_D \tilde{\ph}_n - \tilde{\lambda}_n \tilde{\ph}_n \|_{L^2(D)}}{\|\tilde{\ph}_n\|_{L^2(D)}} \, .
}
In fact, if $A_D^{\even}$ and $A_D^{\odd}$ are the restrictions of $A_D$ to the (invariant) subspaces of $L^2(D)$ consisting of even and odd functions, respectively, then~\eqref{eq:dist0} holds with $\sigma(A_D)$ replaced by $\sigma(A_D^{\even})$ when $n$ is odd, and by $\sigma(A_D^{\odd})$ when $n$ is even.
\end{lemma}


The following result is an immediate consequence of Lemmas~\ref{lem:approxnorm0} and~\ref{lem:dist0}.

\begin{corollary}
\label{cor:dist0}
If $\psi$ is unbounded, for all $n \ge 7$ there is a positive integer $k(n)$ such that
\formula{
 |\tilde{\lambda}_n - \lambda_{k(n)}| & \le \frac{C(a, b, \psi)}{n} \, .
}
\end{corollary}

\begin{lemma}
\label{lem:separation}
Suppose that $\lim_{\xi \to \infty} \xi \psi'(\xi) = \infty$. For $n$ larger than some (integer) constant $C(a, b, \psi)$ the numbers $k(n)$ are distinct. Moreover, for any $\eps > 0$, for $n$ larger than some (integer) constant $C(a, b, \psi, \eps)$,
\formula[eq:separation]{
 \psi((\tilde{\mu}_n - \eps)^2) & < \lambda_{k(n)} < \psi((\tilde{\mu}_n + \eps)^2) .
}
\end{lemma}

\begin{proof}
Let $\eps \in (0, \tfrac{\pi}{4 a})$. For some $\xi_n \in (\tilde{\mu}_n, \tilde{\mu}_n + \eps)$,
\formula{
 \psi((\tilde{\mu}_n + \eps)^2) - \psi(\tilde{\mu}_n^2) = 2 \eps \xi_n \psi'(\xi_n^2) .
}
Since $\xi_n \le \tfrac{n \pi}{2 a} + \eps \le \tfrac{n \pi}{a}$, it follows that
\formula{
 \psi((\tilde{\mu}_n + \eps)^2) - \psi(\tilde{\mu}_n^2) & \ge \frac{2 a \eps \xi_n^2 \psi'(\xi_n^2)}{n \pi} \, .
}
Since $\xi_n \ge \tfrac{(n - 1) \pi}{2 a}$, we have $\lim_{n \to \infty} \xi_n^2 \psi'(\xi_n^2) = \infty$, and so, by Corollary~\ref{cor:dist0}, for $n$ greater than some constant $C(a, b, \psi, \eps)$,
\formula{
 \psi((\tilde{\mu}_n + \eps)^2) - \psi(\tilde{\mu}_n^2) & > |\tilde{\lambda}_n - \lambda_{k(n)}| .
}
Since $\psi$ is concave,
\formula{
 \psi(\tilde{\mu}_n^2) - \psi((\tilde{\mu}_n - \eps)^2) & \ge \psi((\tilde{\mu}_n + \eps)^2) - \psi(\tilde{\mu}_n^2) .
}
Finally, $\tilde{\lambda}_n = \psi(\tilde{\mu}_n^2)$. This proves~\eqref{eq:separation}.

Observe that, by Lemma~\ref{lem:dthetmuest},
\formula{
 a \tilde{\mu}_{n+1} - a \tilde{\mu}_n & = \frac{\pi}{2} + \thet_{\tilde{\mu}_n} - \thet_{\tilde{\mu}_{n+1}} \ge \frac{\pi}{2} - \frac{3}{\tilde{\mu}_n} (\tilde{\mu}_{n+1} - \tilde{\mu}_n) \ge \frac{\pi}{2} - \frac{6 a}{(n - 1) \pi} (\tilde{\mu}_{n+1} - \tilde{\mu}_n) \, ,
}
so that $\tilde{\mu}_{n+1} - \tilde{\mu}_n \ge \tfrac{\pi}{2 a} (1 + \tfrac{6}{(n - 1) \pi})^{-1} \ge \tfrac{\pi}{4 a}$ for $n \ge 3$. The first statement of the lemma follows hence from~\eqref{eq:separation} with $\eps = \tfrac{\pi}{8 a}$.
\end{proof}

\begin{lemma}
\label{lem:cscor}
Suppose that $\lim_{\xi \to \infty} \xi \psi'(\xi) = \infty$. Then $k(n) \ge n$ for infinitely many $n$.
\end{lemma}

\begin{proof}
By Lemma~\ref{lem:separation},
\formula{
 \lambda_{k(n)} \ge \psi((\tilde{\mu}_n - \tfrac{\pi}{16 a})^2)
}
for $n$ large enough. On the other hand, by~\eqref{eq:cs},
\formula{
 \lambda_{n - 1} \le \psi((\tfrac{(n - 1) \pi}{2 a})^2)
}
for all $n \ge 1$. Finally, by Lemma~\ref{lem:thetmulim} and Lemma~\ref{lem:dthetmuest}, $\thet_{\tilde{\mu}_n} < \tfrac{3 \pi}{8} + \tfrac{\pi}{16}$ for infinitely many $n$, and hence
\formula{
 \tilde{\mu}_n - \tfrac{\pi}{16 a} = \tfrac{n \pi}{2 a} - \tfrac{1}{a} \thet_{\tilde{\mu}_n} - \tfrac{\pi}{16 a} > \tfrac{n \pi}{2 a} - (\tfrac{3 \pi}{8 a} + \tfrac{\pi}{16 a}) - \tfrac{\pi}{16 a} = \tfrac{(n - 1) \pi}{2 a}
}
for infinitely many $n$.
\end{proof}

\subsection{Trace estimate}
\label{sec:prtrace}

Recall that the kernel functions of the operators $\exp(-t A)$ and $\exp(-t A_D)$ are denoted by $T(t; x - y)$ and $T_D(t; x, y)$, respectively. Furthermore, $0 \le T_D(t; x, y) \le T(t; x - y)$ for all $t > 0$ and $x, y \in D = (-a, a)$, and the Fourier transform of $T(t; x)$ is $\exp(-t \psi(\xi^2))$. In order to estimate the number of eigenvalues $\lambda_n$ not counted as $\lambda_{k(n)}$ for $n$ large enough, we use the trace estimate method, applied previously in in~\cite[Section~9]{bib:kkms10}, \cite[Section~5]{bib:k12} and~\cite[Step~4 of the proof of Lemma~4.4]{bib:kkm13}, see also~\cite{bib:bk08,bib:k98}.

\begin{lemma}
\label{lem:trace}
Suppose that $\lim_{\xi \to \infty} \xi \psi'(\xi) = \infty$. For $n$ greater than some constant $C(a, b, \psi)$ we have $k(n) = n$.
\end{lemma}

\begin{proof}
Let $\eps = \tfrac{\pi}{6 a}$ and let $N$ be the constant $C(a, b, \psi, \eps)$ in Lemma~\ref{lem:separation}. Define $J = \{ k(n) : n > N \}$ and let $J' = \{j \ge 1 : j \notin J\}$. We claim that it suffices to show that $|J'| \le N$. Indeed, there is $n_0 > N$ such that $k(n_0) = 1 + \max J'$, and $k(n)$ is strictly increasing for $n > N$. It follows that $k(n) = k(n_0) + n - n_0$ for $n \ge n_0$. If $|J'| \le N$, then $k(n_0) = |J'| + (n_0 - N) \le n_0$, so that $k(n) \le n$ for $n \ge n_0$. Since $k(n) \ge n$ infinitely many times by Lemma~\ref{lem:cscor}, necessarily $k(n) = n$ for $n \ge n_0$, as desired.

Let $t > 0$. By the assumption, $\psi(\xi) \ge \tfrac{1}{t} \log \xi - C(t)$ for some constant $C(t)$, and therefore $\exp(-t \psi(\xi^2))$ is integrable. Therefore, $T(t; x)$ is bounded in $x \in \R$. In particular, $T_D(t; x, \cdot)$ is in $L^2(D)$, and so, by Parseval's identity,
\formula{
 \int_{-a}^a \int_{-a}^a (T_D(t; x, y))^2 dy dx & = \int_{-a}^a \sum_{n = 1}^\infty \expr{\int_{-a}^a T_D(t; x, y) \ph_j(y) dy}^2 dx \\
 & = \int_{-a}^a \sum_{j = 1}^\infty e^{-2 \lambda_j t} (\ph_j(x))^2 dx = \sum_{j = 1}^\infty e^{-2 \lambda_j t} .
}
On the other hand, by Plancherel's identity,
\formula{
 \int_{-a}^a \int_{-a}^a (T_D(t; x, y))^2 dy dx & \le 2 a \int_{-\infty}^\infty (T(t; x - y))^2 dy = \frac{2 a}{\pi} \int_0^\infty e^{-2 t \psi(\xi^2)} d\xi .
}
It follows that for all $t > 0$,
\formula[eq:trace:1]{
 \sum_{j = 1}^\infty e^{-\lambda_j t} & \le \frac{2 a}{\pi} \int_0^\infty e^{-t \psi(\xi^2)} d\xi .
}
Observe that
\formula{
 \sum_{j \in J} e^{-\lambda_j t} & = \sum_{n = N}^\infty e^{-\lambda_{k(n)} t} \ge \sum_{n = N + 1}^\infty e^{-\psi((\tilde{\mu}_n + \eps)^2) t} \ge \sum_{n = N}^\infty e^{-\psi((n \pi / (2 a) + \eps)^2) t} .
}
Denote $\xi_n = n \pi / (2 a) + \eps = (n + \tfrac{1}{3}) \pi / (2 a)$. Since $e^{-t \psi(z)}$ is concave in $z > 0$,
\formula{
 \int_{\xi_n}^{\xi_{n+1}} e^{-t \psi(\xi^2)} d\xi & \le \int_{\xi_n}^{\xi_{n+1}} \expr{\frac{\xi_{n+1}^2 - \xi^2}{\xi_{n+1}^2 - \xi_n^2} \, e^{-t \psi(\xi_n^2)} + \frac{\xi^2 - \xi_n^2}{\xi_{n+1}^2 - \xi_n^2} \, e^{-t \psi(\xi_{n+1}^2)}} d\xi \\
 & = \frac{2 \xi_{n+1}^2 - \xi_n \xi_{n+1} - \xi_n^2}{3 (\xi_n + \xi_{n+1})} \, e^{-t \psi(\xi_n^2)} + \frac{\xi_{n+1}^2 + \xi_n \xi_{n+1} - 2 \xi_n^2}{3 (\xi_n + \xi_{n+1})} \, e^{-t \psi(\xi_{n+1}^2)} \\
 & = \frac{\pi}{2 a} \expr{\frac{3 n + 3}{6 n + 5} \, e^{-t \psi(\xi_n^2)} + \frac{3 n + 2}{6 n + 5} \, e^{-t \psi(\xi_{n+1}^2)}} . 
}
Hence,
\formula{
 \frac{2 a}{\pi} \int_{\xi_N}^\infty e^{-t \psi(\xi^2)} d\xi & \le \sum_{n = N}^\infty \expr{\frac{3 n + 3}{6 n + 5} \, e^{-t \psi(\xi_n^2)} + \frac{3 n + 2}{6 n + 5} \, e^{-t \psi(\xi_{n+1}^2)}} \\
 & \le \frac{3 N + 3}{6 N + 5} \, e^{-t \psi(\xi_N^2)} + \sum_{n = N + 1}^\infty e^{-t \psi(\xi_n^2)} \\
 & \le \frac{3 N + 3}{6 N + 5} \, e^{-t \psi(\xi_N^2)} + \sum_{j \in J} e^{-t \lambda_j}
}
(the second inequality is a consequence of $\tfrac{3 n + 2}{6 n + 5} + \tfrac{3(n + 1) + 3}{6(n + 1) + 5} \le 1$, while the last one follows from $\lambda_{k(n)} \le \psi((\tilde{\mu}_n + \eps)^2) \le \psi(\xi_n^2)$ for $n > N$). By~\eqref{eq:trace:1},
\formula{
 \sum_{j \in J'} e^{-\lambda_j t} & \le \frac{2 a}{\pi} \int_0^\infty e^{-t \psi(\xi^2)} d\xi - \sum_{j \in J} e^{-\lambda_j t} \le \frac{2 a}{\pi} \int_0^{\xi_N} e^{-t \psi(\xi^2)} d\xi + \frac{3 N + 3}{6 N + 5} \, e^{-t \psi(\xi_N^2)} \, .
}
Passing to a limit as $t \to 0^+$, we obtain
\formula{
 |J'| & \le \frac{2 a}{\pi} \, \xi_N + \frac{3 N + 3}{6 N + 5} = N + \frac{1}{3} + \frac{3 N + 3}{6 N + 5} < N + 1 .
}
This shows that $|J'| \le N$, as desired.
\end{proof}

\begin{proof}[Proof of Theorem~\ref{th:main}]
By Lemma~\ref{lem:trace}, $k(n) = n$ for $n$ large enough. Hence, by Corollary~\ref{cor:dist0},
\formula{
 \lambda_n & = \tilde{\lambda}_n + O(\tfrac{1}{n}) = \psi(\tilde{\mu}_n^2) + O(\tfrac{1}{n}).
\qedhere
}
\end{proof}

\subsection{Properties of eigenfunctions}
\label{subsec:prop}

As in the previous articles~\cite{bib:kkm13,bib:kkms10,bib:k12}, the intermediate results in the proof of Theorem~\ref{th:main} provide some approximation results for the eigenfunctions. The details of the argument differ slightly from that of~\cite{bib:kkm13,bib:kkms10,bib:k12}, so we sketch the proofs.

\begin{proposition}[{see~\cite[Proposition~1]{bib:k12} and~\cite[Proposition~4.9]{bib:kkm13}}]
\label{prop:l2approx1}
Suppose that $\lim_{\xi \to \infty} \xi \psi'(\xi) = \infty$. With the appropriate choice of the signs of $\ph_n$ and with
\formula{
 \beta_n & = \|\tilde{\ph}_n\|_{L^2(D)}
}
we have $\beta_n = \sqrt{a} + O(\tfrac{1}{n})$ as $n \to \infty$, and
\formula{
 \|\tilde{\ph}_n - \beta_n \, \ph_n\|_{L^2(D)} & = O\expr{\frac{1}{(\tfrac{n \pi}{2 a})^2 \psi'((\tfrac{n \pi}{2 a})^2)}} && \text{as $n \to \infty$.}
}
\end{proposition}

\begin{proof}
By Lemma~\ref{lem:approxnorm0}, indeed $\beta_n = \sqrt{a} + O(\tfrac{1}{n})$. Let $\alpha_{n,j} = \langle\tilde{\ph}_n, \ph_j\rangle_{L^2(D)}$, so that $\tilde{\ph}_n = \sum_{j = 1}^\infty \alpha_{n,j} \ph_j$ in $L^2(D)$. We choose the sign of $\ph_n$ so that $\alpha_{n,n} \ge 0$. We have
\formula{
 \norm{\tilde{\ph}_n - \beta_n \ph_n}_{L^2(D)} & \le \|\tilde{\ph}_n - \alpha_{n,n} \ph_n\|_{L^2(D)} + |\alpha_{n,n} - \beta_n| \\
 & = \|\tilde{\ph}_n - \alpha_{n,n} \ph_n\|_{L^2(D)} + |\|\alpha_{n,n} \ph_n\|_{L^2(D)} - \|\tilde{\ph}_n\|_{L^2(D)}| \\
 & \le 2 \|\tilde{\ph}_n - \alpha_{n,n} \ph_n\|_{L^2(D)} .
}
As in the proof of Lemma~\ref{lem:separation}, for $n$ larger than some constant, if $j \ne n$ and $\eps = \tfrac{\pi}{8 a}$, then
\formula{
 |\lambda_j - \tilde{\lambda}_n| & \ge \max\bigl(\psi((\tilde{\mu}_{n+1} - \eps)^2) - \psi((\tilde{\mu}_n + \eps)^2), \psi((\tilde{\mu}_n - \eps)^2) - \psi((\tilde{\mu}_{n-1} + \eps)^2\bigr) \\
 & \ge 2 \tfrac{(n - 1) \pi}{2 a} \psi'((\tfrac{(n + 1) \pi}{2 a})^2) \cdot (\tfrac{\pi}{2 a} - 2 \eps) \ge \tfrac{1}{C_1} \, \tfrac{n \pi}{2 a} \psi'((\tfrac{n \pi}{2 a})^2) .
}
Therefore,
\formula{
 \|\tilde{\ph}_n - \alpha_{n,n} \ph_n\|^2_{L^2(D)} & = \sum_{j \ne n} \abs{\alpha_{n,j}}^2 \le \frac{C_1}{\tfrac{n \pi}{2 a} \psi'((\tfrac{n \pi}{2 a})^2)} \sum_{j \ne n} (\lambda_j - \tilde{\lambda}_n)^2 \abs{\alpha_{n,j}}^2 \\
 & \le \frac{C_1}{\tfrac{n \pi}{2 a} \psi'((\tfrac{n \pi}{2 a})^2)} \|A_D \tilde{\ph}_n - \tilde{\lambda}_n \tilde{\ph}_n\|_{L^2(D)}^2 \le \frac{C_2(a, b, \psi)}{(\tfrac{n \pi}{2 a})^2 \psi'((\tfrac{n \pi}{2 a})^2)} ,
}
again by Lemma~\ref{lem:approxnorm0}.
\end{proof}

\begin{proposition}[{see~\cite[Proposition~1]{bib:k12} and~\cite[Proposition~1.2]{bib:kkm13}}]
\label{prop:l2approx2}
Suppose that $\lim_{\xi \to \infty} \xi \psi'(\xi) = \infty$. With the appropriate choice of the signs of $\ph_n$ and with
\formula{
 f_n(x) & = \begin{cases} (-1)^{(n-1)/2} \tfrac{1}{\sqrt{a}} \cos(\tilde{\mu}_n x) & \text{when $n$ is odd}, \\ (-1)^{n/2} \tfrac{1}{\sqrt{a}} \sin(\tilde{\mu}_n x) & \text{when $n$ is even}, \end{cases}
}
we have
\formula{
 \|f_n - \ph_n\|_{L^2(D)} & = O\expr{\frac{1}{\sqrt{n}} + \frac{1}{(\tfrac{n \pi}{2 a})^2 \psi'((\tfrac{n \pi}{2 a})^2)}} && \text{as $n \to \infty$.}
}
\end{proposition}

\begin{proof}
Clearly,
\formula{
 \|f_n - \ph_n\|_{L^2(D)} & \le \|f_n - \tfrac{1}{\sqrt{a}} \tilde{\ph}_n\|_{L^2(D)} + \tfrac{1}{\sqrt{a}} \|\tilde{\ph}_n - \beta_n \ph_n\|_{L^2(D)} + |\tfrac{\beta_n}{\sqrt{a}} - 1| \|\ph_n\|_{L^2(D)} .
}
The middle summand is $O(1 / ((\tfrac{n \pi}{2 a})^2 \psi'((\tfrac{n \pi}{2 a})^2)))$, while the last one is $O(\tfrac{1}{n})$. Finally, by the definition~\eqref{eq:phitilde} of $\tilde{\ph}_n$ and the properties of $q(x)$ and $F_\mu(x)$, 
\formula{
 \|\sqrt{a} \, f_n - \tilde{\ph}_n\|_{L^2(D)}^2 & = \int_{-a}^a (q(-x) G_{\tilde{\mu}_n}(a + x) - (-1)^n q(x) G_{\tilde{\mu}_n}(a - x))^2 dx \\
 & \le 4 \int_0^\infty (G_{\tilde{\mu}_n}(s))^2 ds \le 4 G_{\tilde{\mu}_n}(0) \int_0^\infty G_{\tilde{\mu}_n}(s) ds = 4 G_{\tilde{\mu}_n}(0) \laplace G_{\tilde{\mu}_n}(0) .
}
Since $G_\mu(0) = \cos \thet_\lambda \le 1$ and $\laplace G_\mu(0) = I_\mu \le \tfrac{1}{\mu}$ (see~\eqref{eq:lgmu0}), we have
\formula{
 \|\sqrt{a} \, f_n - \tilde{\ph}_n\|_{L^2(D)} & = O(\tfrac{1}{\sqrt{n}}) .
\qedhere
}
\end{proof}

\begin{proposition}[{see~\cite[Proposition~2]{bib:k12} and~\cite[Proposition~1.3]{bib:kkm13}}]
\label{prop:loo}
Suppose that if $\xi_2 > \xi_1 > 1$, then
\formula[eq:scaling]{
 \frac{\psi(\xi_2)}{\psi(\xi_1)} \ge M \expr{\frac{\xi_2}{\xi_1}}^\eps
}
for some $M, \eps > 0$. Suppose in addition that
\formula[eq:threefourths]{
 \liminf_{\xi \to \infty} \xi^{3/4} \, \psi'(\xi) > 0 .
}
Then $\ph_n(x)$ are bounded uniformly in $n \ge 1$ and $x \in (-a, a)$.
\end{proposition}

Condition~\eqref{eq:scaling} is known under various names, including \emph{weak lower scaling condition} and \emph{subregularity}; such a function $\psi$ is also said to have positive \emph{lower Matuszewska index}. We remark that although~\eqref{eq:threefourths} does not imply~\eqref{eq:scaling}, examples of complete Bernstein functions which satisfy~\eqref{eq:threefourths}, but not~\eqref{eq:scaling}, are rather artificial.

\begin{proof}
Observe that $\xi \psi'(\xi)$ diverges to $\infty$ as $\xi \to \infty$, and therefore main results of the present article apply. Furthermore, by~\eqref{eq:scaling}, we have $T(t, 0) \le C_1(\psi) \sqrt{\psi^{-1}(1/t)}$ for $t \le 1$, see, for example, \cite[Theorem~21]{bib:bgr13}.

We have
\formula{
 \|\ph_n\|_{L^\infty(D)} & = e^{\lambda_n t} \|T_D(t) \ph_n\|_{L^\infty(D)} \\
 & \le e^{\lambda_n t} \|T_D(t) (\ph_n - \tfrac{1}{\beta_n} \tilde{\ph}_n)\|_{L^\infty(D)} + e^{\lambda_n t} \tfrac{1}{\beta_n} \|T_D(t) \ph_n\|_{L^\infty(D)} .
}
Since $|\ph_n(x)| \le 2$, the latter term in the right-hand side does not exceed $\tfrac{2}{\beta_n} e^{\lambda_n t}$. For the former one, observe that $|T_D(t) f(x)| \le \|T_D(t, x, \cdot)\|_{L^2(D)} \|f\|_{L^2(D)}$, $T_D(t, x, y) \le T(t, x - y)$, and, by Plancherel's theorem,
\formula{
 \|T(t, \cdot)\|_{L^2(\R)}^2 & = \frac{1}{2 \pi} \int_{-\infty}^\infty (e^{-t \psi(\xi^2)})^2 d\xi = T(2 t, 0) .
}
Finally, $T(2 t, 0) \le C_1(\psi) \sqrt{\psi^{-1}(1/(2t))} \le C_1(\psi) \sqrt{\psi^{-1}(1/t)}$ when $t \le 1$. Therefore, with $t = \tfrac{1}{\lambda_n}$,
\formula{
 \|\ph_n\|_{L^\infty(D)} & \le \tfrac{e}{\beta_n} (C_1(\psi))^{1/2} (\psi^{-1}(\lambda_n))^{1/4} \|\beta_n \ph_n - \tilde{\ph}_n\|_{L^2(D)} + \tfrac{2 e}{\beta_n} .
}
In the right-hand side, $\beta_n = O(1)$, $\psi^{-1}(\lambda_n) \le (\tfrac{n \pi}{2 a})^2$ (by~\eqref{eq:cs}), and, by Lemma~\ref{lem:approxnorm0},
\formula{
 \|\beta_n \ph_n - \tilde{\ph}_n\|_{L^2(D)} & = O\expr{\frac{1}{(\tfrac{n \pi}{2 a})^2 \psi'((\tfrac{n \pi}{2 a})^2)}} .
\qedhere
}
\end{proof}

%
%

\section*{Appendix}

\begin{proof}[Proof of Lemma~\ref{lem:cbfm}]
Let $x > 0$, $0 < \eps < \tfrac{1}{2} C_1$ and $y > 0$, and denote for simplicity $\xi = -e^{-i \eps} x$. By the representation~\eqref{eq:cbf} of the complete Bernstein function $f$ and Fubini, we have
\formula[eq:cbfm:main]{
\begin{aligned}
 \int_0^\infty \imag(f(\xi) g(-\xi)) h(x) dx & = c \int_0^\infty \imag(\xi g(-\xi)) h(x) dx + \tilde{c} \int_0^\infty \imag(g(-\xi)) h(x) dx \\
 & \hspace*{5em} + \frac{1}{\pi} \int_{(0, \infty)} \int_0^\infty \imag\expr{\frac{\xi g(-\xi)}{\xi + z}} h(x) dx \, \frac{m(dz)}{z}
\end{aligned}
}
(an estimate which allows us to use Fubini is shown below). Our goal is to provide estimates for the integrands and find their pointwise limits as $\eps \to 0^+$ in order to apply dominated convergence.

For the first integral in the right-hand side of~\eqref{eq:cbfm:main}, we simply use $|\xi g(-\xi) h(x)| \le x G(x) H(x)$, integrability of $x G(x) H(x)$ and $\imag(\xi g(-\xi)) \to 0$ as $\eps \to 0^+$. By dominated convergence, the limit as $\eps \to 0^+$ of the first integral in the right-hand side of~\eqref{eq:cbfm:main} is zero. Similarly, $|g(-\xi) h(\xi)| \le G(x) H(x)$, $G(x) H(x)$ is integrable and $\imag(g(-\xi)) \to 0$ as $\eps \to 0^+$, and so also the second integral in the right-hand side of~\eqref{eq:cbfm:main} converges to zero as $\eps \to 0^+$.

To estimate the last integral in the right-hand side of~\eqref{eq:cbfm:main}, we consider separately two cases. When $x \le \tfrac{y}{2}$ or $x \ge 2 y$, we have
\formula{
 \abs{-\frac{\xi}{\xi + y} \, g(-\xi)} & \le \frac{1}{|x - y|} \, x G(x) \le \frac{3}{x + y} \, x G(x) \\
 & \le 3 \min(1, x y^{-1}) G(x) \le 3 \min(1, y^{-1}) (1 + x) G(x) ,
}
so that by dominated convergence,
\formula[eq:cbfm3]{
\begin{gathered}
 \expr{\int_0^{y/2} + \int_{2 y}^\infty} \abs{\imag \expr{\frac{\xi}{\xi + y} \, g(-\xi)} h(x)} dx \le 3 C_3 \min(1, y^{-1}) \, , \\
 \lim_{\eps \to 0^+} \expr{\int_0^{y/2} + \int_{2 y}^\infty} \imag \expr{\frac{\xi}{\xi + y} \, g(-\xi)} h(x) dx = 0 .
\end{gathered}
}
When $\tfrac{y}{2} < x < 2 y$, we need a more careful estimate. Observe that
\formula{
 \frac{\xi g(-\xi)}{\xi + y} & = \frac{y g(y) - (-\xi) g(-\xi)}{y - (-\xi)} - \frac{y g(y)}{\xi + y} \, .
}
The estimate for $g$ and Cauchy's integral formula for $g'$ easily give
\formula{
 |g'(z)| & \le C_4 y^{-1} G(y)
}
in $\{z \in \C : |\Arg z| < \tfrac{1}{2} C_1, \, y/2 \le |z| \le 2 y\}$, with $C_4 = 4 C_1^{-1}$. By the mean value theorem,
\formula{
 \abs{\frac{y g(y) - (-\xi) g(-\xi)}{y - (-\xi)}} & \le C_4 y^{-1} G(y)
}
when $\tfrac{y}{2} \le x \le 2 y$, and therefore, by dominated convergence,
\formula[eq:cbfm2]{
\begin{gathered}
 \int_{y/2}^{2 y} \abs{\imag \expr{\frac{y g(y) - (-\xi) g(-\xi)}{y - (-\xi)}} h(x)} dx \le \tfrac{3}{2} C_4 y^{-1} G(y) H(y) \le \tfrac{3}{2} C_2 C_4 \min(1, y^{-1}) , \\
 \lim_{\eps \to 0^+} \int_{y/2}^{2 y} \imag \expr{\frac{y g(y) - (-\xi) g(-\xi)}{y - (-\xi)}} h(x) dx = 0 .
\end{gathered}
}
Finally, if $P_t(s)$ and $Q_t(s)$ denote the (classical) Poisson and conjugate Poisson kernels for the half-plane, then
\formula{
 \imag \expr{-\frac{1}{\xi + y}} 
 & = \pi \cos(\eps) P_{y \sin \eps}(x - y \cos \eps) + \pi \sin(\eps) Q_{y \sin \eps}(x - y \cos \eps) .
}
Clearly, $P_{y \sin \eps}(x - y \cos \eps) \ind_{(y/2, 2y)}(x) dx$ converges weakly to $\delta_y(x)$, and therefore
\formula{
\begin{gathered}
 \int_{y/2}^{2y} \abs{\pi \cos(\eps) P_{y \sin \eps}(x - y \cos \eps) y g(y) h(x)} dx \le \pi y g(y) H(y) \le C_2 \pi \min(1, y^{-1}) , \\
 \lim_{\eps \to 0^+} \int_{y/2}^{2y} \pi \cos(\eps) P_{y \sin \eps}(x - y \cos \eps) y g(y) h(x) dx = \pi y g(y) h(y) .
\end{gathered}
}
Furthermore, $|t Q_t(s)| \le \tfrac{1}{\pi}$ and $t Q_t(s) \to 0$ as $t \to 0^+$, and hence, by dominated convergence,
\formula{
\begin{gathered}
 \int_{y/2}^{2y} \abs{\pi \sin(\eps) Q_{y \sin \eps}(x - y \cos \eps) y g(y) h(x)} dx \le \tfrac{3}{2} y g(y) H(y) \le \tfrac{3}{2} C_2 \min(1, y^{-1}) , \\
 \lim_{\eps \to 0^+} \int_{y/2}^{2y} \pi \sin(\eps) Q_{y \sin \eps}(x - y \cos \eps) y g(y) h(x) dx = 0 .
\end{gathered}
}
We have thus proved that
\formula[eq:cbfm1]{
\begin{gathered}
 \int_{y/2}^{2y} \abs{\imag \expr{-\frac{y g(y)}{\xi + y}} h(x)} dx \le C_2 (\pi + \tfrac{3}{2}) \min(1, y^{-1}) , \\
 \lim_{\eps \to 0^+} \int_{y/2}^{2y} \imag \expr{-\frac{y g(y)}{\xi + y}} h(x) dx = \pi y g(y) h(y) .
\end{gathered}
}
Due to estimates~\eqref{eq:cbfm3}, \eqref{eq:cbfm2} and \eqref{eq:cbfm1}, as well as the integrability condition on $m$, indeed we could use Fubini in~\eqref{eq:cbfm:main}. The same estimates allow us to use dominated convergence in the limit as $\eps \to 0^+$. We conclude that
\formula{
 \lim_{\eps \to 0^+} \int_0^\infty \imag (f(\xi) g(-\xi)) h(x) dx & = \pi \int_{(0, \infty)} \int_0^\infty g(y) h(y) m(dy) .
}
This proves the first equality in~\eqref{eq:fgh}. The other one follows by replacing the pair $g(z)$, $h(x)$ with $1$ and $g(x) h(x)$.
\end{proof}

%
%


%
%


\begin{thebibliography}{99}

\bibitem{bib:a04}
D.~Applebaum,
\emph{L{\'e}vy Processes and Stochastic Calculus}.
Cambridge University Press, Cambridge, 2004.


\bibitem{bib:bk08}
R.~Ba{\~n}uelos, T.~Kulczycki,
\emph{Trace estimates for stable processes}.
Probab. Theory Relat. Fields 142(3--4) (2008) 313--338.

\bibitem{bib:bks09}
R.~Ba{\~n}uelos, T.~Kulczycki,
\emph{On the traces of symmetric stable processes on Lipschitz domains}.
J.~Funct. Anal. 257(10) (2009): 3329--3352.

\bibitem{bib:bmn14}
R.~Ba{\~n}uelos, J.~B.~Mijena, E.~Nane,
\emph{Two-term trace estimates for relativistic stable processes}.
J.~Math. Anal. Appl. 410(2) (2014): 837--846.


\bibitem{bib:bg59}
R.~M.~Blumenthal, R.~K.~Getoor,
\emph{The asymptotic distribution of the eigenvalues for a class of Markov operators}.
Pacific J.~Math. 9(2) (1959) 399--408.






\bibitem{bib:bgr13}
K.~Bogdan, T.~Grzywny, M.~Ryznar,
\emph{Density and tails of unimodal convolution semigroups}.
J.~Funct. Anal. 266(6) (2014): 3543--3571.


\bibitem{bib:bs15}
K.~Bogdan, B.~Siudeja,
\emph{Trace estimates for unimodal L\'evy processes}.
Preprint, 2015, arXiv:1503.09126v1.





\bibitem{bib:cs05}
Z.~Q.~Chen, R.~Song,
\emph{Two sided eigenvalue estimates for subordinate Brownian motion in bounded domains}.
J.~Funct. Anal. 226 (2005) 90--113.


\bibitem{bib:c85}
K.~L.~Chung,
\emph{Doubly-Feller process with multiplicative functional}.
In: \emph{Seminar on Stochastic Processes, 1985} (Gainesville, Fla., 1985),
Progr. Probab. Statist.~12, Birkh\"{a}user, Boston, Boston, MA., 1986, 63--78.


\bibitem{bib:d70}
A.~M.~J.~Davis,
\emph{Waves in the presence of an infinite dock with gap}.
J.~Inst. Math. Appl. 6 (1970): 141--156.




\bibitem{bib:dkk15}
B.~Dyda, A.~Kuznetsov, and M.~Kwa{\'s}nicki.
\emph{Eigenvalues of the fractional Laplace operator in the unit ball}.
In preparation, 2015.



\bibitem{bib:fg14}
R.~Frank, L.~Geisinger,
\emph{Refined semiclassical asymptotics for fractional powers of the Laplace operator}.
J.~Reine Angew. Math., to appear.

\bibitem{bib:fot10}
M.~Fukushima, Y.~Osima, M.~Takeda,
\emph{Dirichlet Forms and Symmetric Markov Processes}.
De Gruyter Studies in Mathematics 19, de Gruyter, Berlin, Boston, 2010.



\bibitem{bib:j01}
N.~Jacob,
\emph{Pseudo Differential Operators and Markov Processes}, Vol. 1.
Imperial College Press, London, 2001


\bibitem{bib:jk15}
T.~Juszczyszyn, M.~Kwa\'snicki,
\emph{Hitting times of points for symmetric L\'evy processes with completely monotone jumps}.
Electron. J. Probab. 20 (2015): 48:1--24.

\bibitem{bib:kkm13}
K.~Kaleta, M.~Kwa\'snicki, J.~Ma{\l}ecki,
\emph{One-dimensional quasi-relativistic particle in the box}.
Rev. Math. Phys. 25(8) (2013) 1350014.




\bibitem{bib:k98}
T.~Kulczycki,
\emph{Intrinsic ultracontractivity for symmetric stable processes}.
Bull. Polish Acad. Sci. Math. 46(3) (1998) 325--334.

\bibitem{bib:kkms10}
T.~Kulczycki, M.~Kwa{\'s}nicki, J.~Ma{\l}ecki, A.~St{\'o}s,
\emph{Spectral Properties of the Cauchy Process on Half-line and Interval}.
Proc. London Math. Soc. 101(2) (2010) 589--622.


\bibitem{bib:k11}
M.~Kwa{\'s}nicki,
\emph{Spectral analysis of subordinate Brownian motions on the half-line}.
Studia Math. 206(3) (2011) 211--271.

\bibitem{bib:k12}
M.~Kwa{\'s}nicki,
\emph{Eigenvalues of the fractional Laplace operator in the interval}.
J.~Funct. Anal. 262 (2012) 2379--2402.

\bibitem{bib:k12a}
M.~Kwa{\'s}nicki,
\emph{Spectral theory for one-dimensional symmetric Lévy processes killed upon hitting the origin}.
Electron. J. Probab. 17 (2012), no. 83, p. 1--29.

\bibitem{bib:kmr13}
M.~Kwa{\'s}nicki, J.~Ma{\l}ecki, M.~Ryznar,
\emph{First passage times for subordinate Brownian motions}.
Stoch. Proc. Appl 123 (2013) 1820--1850.


\bibitem{bib:ps14}
H.~Park, R.~Song,
\emph{Trace Estimates for Relativistic Stable Processes}.
Potential Anal. 41(4) (2014): 1273--1291.



\bibitem{bib:s99}
K.~Sato,
\emph{L{\'e}vy Processes and Infinitely Divisible Distributions}.
Cambridge Univ. Press, Cambridge, 1999.

\bibitem{bib:ssv10}
R.~Schilling, R.~Song, Z.~Vondra\v{c}ek,
\emph{Bernstein Functions: Theory and Applications}.
De Gruyter, Studies in Math. 37, Berlin, Boston, 2012.



\end{thebibliography}
\end{document}